\documentclass[11pt]{article} 
\usepackage[utf8]{inputenc}
\usepackage{graphicx}
\usepackage[text={6in,8.5in},centering]{geometry}
\usepackage{amssymb}
\usepackage{fourier}
\usepackage{mathrsfs}
\usepackage{mathtools}
\usepackage{framed}
\usepackage[remark, details]{optional}
\usepackage{amsthm}
\usepackage{algorithm}
\usepackage{thm-restate}
\usepackage{algorithmic}
\usepackage{enumerate}
\makeatletter
\newcommand*{\DeclareMathOperators}[1]{%
	\@for\@tmp:=#1\do{%
		\edef\@tmp{%
			\@firstofone{\expandafter\zap@space\@tmp} \@empty
		}%
		\ifx\@tmp\@empty
		\else
		\expandafter\DeclareMathOperator
		\csname\@tmp\expandafter\endcsname\expandafter{\@tmp}%
		\wlog{* Math operator "\@backslashchar \@tmp" defined.}%
		\fi
	}%
}
\makeatother

\newcommand{\si}{\sigma}

\renewcommand{\th}{\theta}

\newcommand{\ga}{\gamma}

\newcommand{\eps}{\varepsilon}
\renewcommand{\epsilon}{\varepsilon}

\renewcommand{\phi}{\varphi}

\newcommand{\scr}[1]{{\mathcal #1}}

\newcommand{\EE}{\mathbb{E}}

\newcommand{\RR}{\mathbb{R}}

\newcommand{\re}{\mathbb{R}}
\newcommand{\TT}{\mathbb{T}}
\DeclareMathOperator{\T}{\mathbb{T}}
\newcommand{\ZZ}{\mathbb{Z}}

\newcommand{\NN}{\mathbb{N}}

\newcommand{\bem}{\begin{bmatrix}}
\newcommand{\enm}{\end{bmatrix}}

\newcommand{\e}{\mathrm{e}}

\renewcommand{\P}{\ensuremath{{\mathrm P}}}

\newcommand{\var}[1]{\ensuremath{{\operatorname{Var}}\left( #1 \right)}}
\newcommand{\cov}[2]{\ensuremath{{\operatorname{Cov}}\left( #1 , #2 \right)}}

\newcommand{\Id}{\mathrm I}
\newcommand{\ind}{{\mathbf{1}}}
\opt{bbm}{
\renewcommand{\ind}{{\mathbbm{1}}}
}
\opt{bbold}{
\renewcommand{\ind}{\mathbbold{1}}
}

\newcommand{\abs}[1]{\left|#1\right|}

\newcommand{\dd}{{\,\mathrm d}}
\providecommand{\I}{\mathrm{I}}

\newcount\colveccount
\newcommand*\colvec[1]{
        \global\colveccount#1
        \begin{pmatrix}
        \colvecnext
}
\def\colvecnext#1{
        #1
        \global\advance\colveccount-1
        \ifnum\colveccount>0
                \\
                \expandafter\colvecnext
        \else
                \end{pmatrix}
        \fi
}

\newcommand{\qtext}[1]{\quad\text{#1}\quad}

\usepackage{color}
\usepackage{natbib}

\usepackage{varioref}
\usepackage[pdftex, breaklinks]{hyperref}
\hypersetup{
    colorlinks=true, 
    linkcolor=darkblue,  
    citecolor = darkgreen,
    bookmarksnumbered=true
    }
\usepackage{cleveref}
\newtheorem{thm}{Theorem}
\newtheorem{lemma}[thm]{Lemma}

\theoremstyle{definition}
\newtheorem{prop}[thm]{Proposition}
\newtheorem{rem}[thm]{Remark}

\newtheorem{ass}[thm]{Assumption}

\title{
Adaptive nonparametric drift estimation for diffusion processes  using Faber-Schauder expansions 
}
\author{Frank van der Meulen\footnote{TU Delft, Mekelweg 4, 2628 CD Delft, The Netherlands, E-mail address: \url{f.h.van der meulen@tudelft.nl}.} \\ Moritz Schauer\footnote{Leiden University, Niels Bohrweg 1, 2333 CA Leiden, The Netherlands, 	E-mail address: \url{m.r.schauer@math.leidenuniv.nl}.
}\\ Jan van Waaij\footnote{Korteweg-de Vries Institute for Mathematics, Science Park 107, 1098 XG Amsterdam, The Netherlands, E-mail address: \url{j.vanwaaij@uva.nl}.
}
}

\DeclareMathOperators{diag, diagm, spn, trace, supp, Var}
\newcommand{\cancel}[1]{}

\newcommand{\sP}{\mathcal{P}}
\definecolor{violet}{rgb}{0.3,0.0, 0.55}

\let\originalleft\left
\let\originalright\right
\renewcommand{\left}{\mathopen{}\mathclose\bgroup\originalleft}
\renewcommand{\right}{\aftergroup\egroup\originalright}

\usepackage{stmaryrd} 

\definecolor{detailsc}{rgb}{0,0.5,0.51}
\definecolor{remarkc}{rgb}{0,0.5,0.5}

\definecolor{darkblue}{rgb}{0.0,0.0, 0.65}
\definecolor{darkgreen}{rgb}{0.0,0.45,0.0}

\newcommand{\remark}[1]{\opt{remark}{{\small {\color{remarkc} \begin{leftbar}\raggedright#1\end{leftbar} }}}}
\newcommand{\todo}[1][inline]{\remark}
\begin{document}\maketitle%
\begin{abstract}

We consider the problem of nonparametric estimation of the drift of a continuously observed one-dimensional diffusion with periodic drift. Motivated by computational considerations, \cite{vdMeulen} defined a prior on the drift as   a randomly truncated and randomly scaled  Faber-Schauder series expansion with Gaussian coefficients.  We study the behaviour of the posterior obtained from this prior from a frequentist asymptotic point of view. If the true data generating drift is smooth, it is proved that the posterior  is adaptive with posterior contraction rates for the $L_2$-norm that are optimal up to a log factor. Contraction rates in $L_p$-norms with $p\in (2,\infty]$ are derived as well. 

\end{abstract}

\section{Introduction}

Assume continuous time observations $X^T=\left\{X_t,\, :t\in[0,T]\right\}$ from a diffusion process $X$ defined as  (weak) solution to the  stochastic differential equation (sde)
\begin{equation}\label{eq:sde}
\dd X_t=b_0(X_t) \dd t+\dd W_t, \qquad X_0=x_0.
\end{equation}
Here $W$ is a Brownian Motion and the drift $b_0$ is assumed to be  a real-valued  measurable function on the real line that is $1$-periodic and square integrable on $[0,1]$. The assumed periodicity implies that we can alternatively view the process $X$ as a diffusion on the circle. This model has been used for dynamic modelling of angles, see for instance \cite{Yvo} and \cite{Hindriks}. 

We are interested in nonparametric adaptive  estimation of the drift. This problem has recently been studied by multiple authors. \cite{Spokoiny} proposed  a locally linear smoother with a data-driven bandwidth choice that  is rate
adaptive with respect to \(|b''(x)|\) for all $x$ and optimal up to a log factors. Interestingly, the result is non-asymptotic and does not require ergodicity. \cite{DalKut} and \cite{Dalalyan} consider ergodic diffusions and construct estimators that are  asymptotically minimax and adaptive under Sobolev smoothness of the drift. Their results were  extended  to the multidimensional case by  \cite{Strauch}.

In this paper we focus on Bayesian nonparametric estimation, a paradigm that has become increasingly popular over the past two decades. An overview of some advances of Bayesian nonparametric estimation for  diffusion processes is given in \cite{vZanten}.

The Bayesian approach requires the specification of a prior.  Ideally, the prior on the drift is chosen such that drawing from the posterior is  computationally efficient while at the same time ensuring that the resulting inference has good theoretical properties. which is  quantified by a contraction rate. This is a rate for which we can shrink balls around the true parameter value, while maintaining most of the posterior mass. More formally, if
  $d$ is  a semimetric on the space of drift functions,  a contraction rate $\eps_T$  is a sequence of positive numbers  $\eps_T\downarrow 0$ for which the posterior mass of the  balls $\{b\,:\, d(b,b_0)\le \eps_T\}$ converges in probability to $1$ as $T\to \infty$, under the law of $X$ with drift  $b_0$. For a general discussion on contraction rates,  see for instance \cite{ghosal2000} and \cite{ghosal2007}.

 For diffusions, the problem of deriving optimal posterior convergence rates has been studied recently under  the additional assumption that the drift integrates to zero, $\int_0^1 b_0(x) d x =0$.  In \cite{Pokern} a mean zero Gaussian process prior is proposed together with an algorithm to sample from the posterior. The precision operator (inverse covariance operator) of the proposed Gaussian process is given by $\eta\left((-\Delta)^{\alpha+1/2} + \kappa I\right)$, where $\Delta$ is the one-dimensional Laplacian, $I$ is the identity operator, $\eta, \kappa >0$ and $\alpha+1/2 \in \{2,3,\ldots\}$. A first consistency result was shown in \cite{PokernStuartvanZanten}.
 
  In \cite{waaijzanten2015} it was shown that this rate result can be improved upon for a slightly more general class  of priors on the drift. More specifically, in this paper the authors consider a prior which is defined as 
\begin{equation}\label{eq:prior-wz} b = L \sum_{k=1}^\infty k^{-1/2-\alpha} \phi_k Z_k, 
\end{equation}
where $\phi_{2k}(x)=\sqrt{2} \cos(2\pi k x)$, $\phi_{2k-1}(x)=\sqrt{2} \sin(2\pi k x)$ are the standard Fourier series basis functions, $\{Z_k\}$ is a sequence of independent standard normally distributed random variables and \(\alpha\) is positive. It is shown that when $L$ and $\alpha$ are fixed and $b_0$ is assumed to be $\alpha$-Sobolev smooth, then the optimal  posterior rate of contraction, $T^{-\alpha/(1+2\alpha)}$, is obtained. Note that this result is nonadaptive, as the regularity of the prior must  match the regularity of $b_0$. 
For obtaining optimal posterior contraction rates for  the full range of possible regularities of the drift, two options are investigated: endowing either $L$ or $\alpha$ with a hyperprior. Only the second option results in the desired adaptivity over all possible regularities. 

While the prior in \eqref{eq:prior-wz} (with additional prior on $\alpha$) has good asymptotic properties,  from a computational point of view the infinite series expansion is inconvenient. Clearly, in any implementation this expansion needs to be truncated. Random truncation of a series  expansion is a well known method for defining priors in Bayesian nonparametrics, see for instance \cite{ghosalshen2015}.  
 Exactly this idea was exploited in \cite{vdMeulen}, where the prior is defined  as the law of the random function
  
\begin{equation}\label{eq:generalprior}
b^{R,S}=  S Z_1 \psi_1+S\sum_{j=0}^R\sum_{k=1}^{2^{j}}Z_{jk}\psi_{jk},
\end{equation}
where the functions $\psi_{jk}$ constitute the Faber-Schauder basis (see   \cref{fig:schauderwav}). 
\begin{figure}%
	\begin{center}
		\includegraphics[width=8cm]{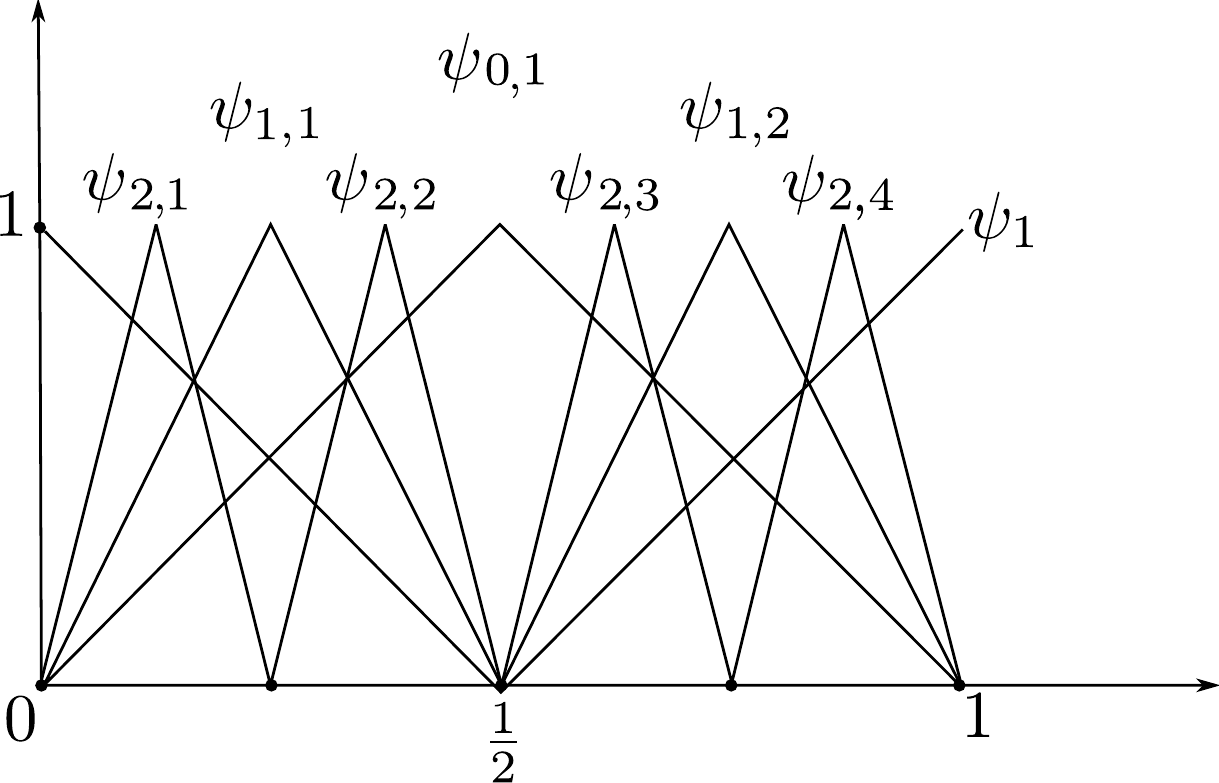}%
	\end{center}
	\caption{Elements $\psi_1$ and $\psi_{j,k}$, $0 \le j \le 2$ of the Faber-Schauder basis}%
	\label{fig:schauderwav}%
\end{figure}
These functions feature prominently in the L\'evy-Ciesielski construction of Brownian motion (see for instance \cite[paragraph 10.1]{BhattacharyaWaymire2007}).
 The prior coefficients $Z_{jk}$ are equipped with a Gaussian distribution, and  the truncation level $R$ and the scaling factor  $S$ are equipped with independent priors. Truncation in absence of scaling increases the apparent smoothness of the prior (as illustrated for deterministic truncation by example 4.5 in \cite{vdVaartvZanten}), whereas scaling by a number $\ge 1$ decreases the apparent smoothness. (Scaling with a number $\le 1$ only increases the apparent smoothness to a limited extent, see for example \cite{KnapikvdVaartvZanten}.)  
 
 The simplest type of prior is obtained by taking the coefficients $Z_{jk}$  independent. We do however also consider the prior that is obtained by first expanding a periodic Ornstein-Uhlenbeck process into the Faber-Schauder basis, followed by random scaling and truncation. We will explain that specific stationarity properties of this prior make it a natural choice. 

Draws from the posterior can be computed using a reversible jump  Markov Chain Monte Carlo (MCMC) algorithm (cf.\ \cite{vdMeulen}). For both types of priors,  fast computation is facilitated by  leveraging inherent sparsity properties stemming from the compact support of the functions $\psi_{jk}$. In the discussion of \cite{vdMeulen} it was argued that  inclusion of both the scaling and random truncation in the prior is beneficial. However, this claim was only supported by simulations results. 

{\it In this paper we support this claim theoretically  by proving adaptive contraction rates of the posterior distribution in case the prior \eqref{eq:generalprior} is used.}
 We start from a general result in  \cite{vdMeulenvdVaartvZanten} on  Brownian semimartingale models,  which we adapt to our setting. Here we take into account that as the drift is assumed to be one-periodic,  information accumulates in a different way  compared to (general) ergodic diffusions.  Subsequently we verify that  the resulting  prior mass, remaining mass and entropy conditions appearing in this adapted result are satisfied  for the prior defined in equation \eqref{eq:generalprior}. 
  An application of our results  shows that  if  the true drift function is  \(B_{\infty,\infty}^\beta\)-Besov smooth, \(\beta\in(0,2)\), then by appropriate choice of the variances of $Z_{jk}$, as well as  the priors on  $R$ and $S$, the posterior for the drift $b$ contracts at the rate $(T/\log T)^{-\beta/(1+2\beta)}$ around the true drift in the $L_2$-norm. Up to the log factor this rate is minimax-optimal (See for instance \cite[][Theorem 4.48]{Kutoyants}). Moreover, it is adaptive: the prior does not depend on $\beta$. In case the true drift has Besov-smoothness greater than or equal to $2$, our method guarantees contraction rates equal to essentially $T^{-2/5}$ (corresponding to $\beta=2$). A further application of our results shows that for  $L_p$-norms we obtain contraction rate  $T^{-(\beta-1/2+1/p)/(1+2\beta)}$, up to log-factors.



The paper is organised as follows. In the next section we give a precise definition of the prior. In \cref{sec:general-theorem}  a general contraction result for the class of diffusion processes considered here is derived.  Our main result on  posterior contraction for  $L^p$-norms with $p\ge 2$ is presented  in \cref{sec:theoremonposteriorcontractionrates}. Many results of this paper concern general properties of the prior and their application is not confined to drift estimation of diffusion processes. To illustrate this, we    show in \cref{sec:auxiliaryresults}  how these results can  easily be adapted to  nonparametric regression and nonparametric density estimation. Proofs are gathered in section \ref{sec:proof}. The appendix contains a couple of technical results.

\section{Prior construction}\label{sec:prior}\label{subsec:stationaryreciprocalprior}

\subsection{Model and posterior}

Let \[L^2(\TT)=\left\{b : \RR \to \RR\:\: \bigg|\:\: \int_0^1 b(x)^2 \dd x <\infty\:\:\text{and}\:\:  b \:\:\text{is $1$-periodic}\right\}\]
 be the space of square integrable $1$-periodic functions.
\begin{lemma}\label{lemma:uniquesolutiontosde}
If $b_0 \in L^2(\T),$ then the  SDE  \cref{eq:sde} has a unique weak solution.
\end{lemma}
The proof is in section \ref{subsec:prooflem1}. 
 
For $b\in L^2(\TT)$, let $P^{b}=P^{b,T}$ denote the law of the process $X^T$ generated by \cref{eq:sde} when \(b_0\) is replaced by \(b\). If $P^0$ denotes the law of $X^T$ when the drift is zero, 
then $P^b$ is absolutely continuous with respect to $P^0$ with Radon-Nikodym density \begin{equation}\label{eq:lik}
p_b\left(X^T\right) =\exp\left(\int_0^Tb(X_t)\dd X_t-\frac{1}{2}\int_0^T b^2(X_t)\dd t\right).
\end{equation}  Given a prior \(\Pi\) on  $L^2(\TT)$ and path $X^T$ from \eqref{eq:sde},  the posterior is given by
\begin{equation}\label{eq:posterior}
\Pi(b \in A\mid X^T) = \frac{\int_A p_b(X^T)\,\Pi(\dd b)}{\int p_b(X^T)\,\Pi(\dd b)},
\end{equation}
where $A$ is Borel set  of  $L^2(\TT)$. These assertions are verified as part of the proof of \cref{thm-general}.

\subsection{Motivating the choice of prior}
We are interested in randomly truncated, scaled series priors that  simultaneously  enable a fast algorithm for obtaining draws from the posterior and 
 enjoy good contraction rates.

To explain what we mean by the first item, consider first a prior that is a {\it finite} series prior. 
 Let $\{\psi_1,\ldots, \psi_r\}$ denote basis functions and $Z=(Z_1,\ldots, Z_r)$ a  mean zero Gaussian random vector with precision matrix $\Gamma$. Assume that the prior for $b$ is given by $b=\sum_{i=1}^r Z_i \psi_i$.
By conjugacy, it follows   that  
$ Z  \mid   X^T \sim \rm N(W^{-1}\mu, W^{-1})$, 
where 
$W= G +  \Gamma$, 
\begin{equation}\label{eq:sigma}
\mu_{i} = \int_0^T \psi_{i}(X_t) \dd X_t\quad\text{and}\quad G_{i, i'} = \int_0^T \psi_{i}(X_t)\psi_{i'}(X_t) \dd t
\end{equation}
 for $i, i' \in \{1,\ldots,r\}$, cf.\  \cite[Lemma 1]{vdMeulen}. The matrix $G$ is referred to as the Grammian. 
From these expressions it follows that it is computationally advantageous to exploit {\it compactly supported} basis functions. 
Whenever $\psi_{i}$ and $\psi_{i'}$ have nonoverlapping supports, we have $G_{i, i'}=0$.
Depending on the choice of such basis functions, the Grammian $G$ will have a specific sparsity structure (a set of index pairs $(i,i')$ such that $G_{i,i'} = 0$, independently of $X^T$.) This sparsity structure is inherited by $W$ as long as the sparsity structure of the prior precision matrix matches that of $G$. 

In the next section we make a specific choice for the basis functions and the prior precision matrix $\Gamma$. 

\subsection{Definition of the prior}\label{subsec:priordef}
Define  the ``hat'' function $\Lambda$ by 
$\Lambda(x) = (2x)\ind_{[0,\frac12)}(x) + 2(1-x) \ind_{[\frac12,1]}(x)$.
The Faber-Schauder basis functions are given by
\[
\psi_{j, k}(x) = \Lambda(2^{j}x - k+1), \quad j\ge 0,  \quad k=1,\ldots, 2^{j}
\]

Let 
\[
\psi_1  = \left(\psi_{0,1}(x-\tfrac12) + \psi_{0,1}(x+\tfrac12)\right)\I_{[0,1]}(x).
\]
In  figure \ref{fig:schauderwav} we have plotted $\psi_1$ together with $\psi_{j,k}$ where $j \in \{0, 1, 2\}$.

We define our prior as in \eqref{eq:generalprior} with Gaussian coefficients $Z_1$ and $Z_{jk}$, where the truncation level $R$ and the scaling factor $S$ are equipped with (hyper)priors. We extend $b$ periodically if we want to consider $b$ as function on the real line. If we identify  the double index $(j,k)$ in \eqref{eq:generalprior} with the single index $i = 2^{j}+k$, then we can write $b^{R,S} = S \sum_{i=1}^{2^{R+1}} \psi_i Z_i$. Let
   \[ \ell(i) = \begin{cases} 0 & \quad \text{if} \quad i\in \{1, 2\} \\
   j & \quad \text{if} \quad  i\in  \left\{2^j +1,\ldots, 2^{j+1}\right\} \quad\text{and}\: j\ge 1 \end{cases}. \]
 We say that $\psi_i$ belongs to level  $j\ge 0$ if $\ell(i) = j$. Thus both \(\psi_1\) and \(\psi_{0,1}\) belong to level 0, which is convenient for notational purposes. For levels \(j\ge1\) the basis functions are per level orthogonal with essentially disjoint support.  
Define for $r\in \{0,1,\ldots\}$  \[ \scr{I}_r = \left\{i\,:\, \ell(i)\le r\}= \{1,2,\ldots,2^{r+1}\right\} . \] 

Let  $A = (\operatorname{Cov}(Z_i, Z_{i'}))_{i,i'\in\NN}$ and define its finite-dimensional restriction by 
$A^r=(A_{ii'})_{i, i' \in \scr{I}_r}$. If we denote  $Z^r=\{Z_i,\, i \in \scr{I}_r\}$, and assume that \(Z^r\) is multivariate normally distributed with mean zero and covariance matrix \(A^r\), then the prior  has the following hierarchy 
\begin{align}
b\mid R, S, Z^R &= S\sum_{i \in \scr{I}_R} Z_i\psi_i\label{priorI}\\
Z^R \mid R  &\sim \rm N(0, A^R)\label{priorII}\\
(R,S) &\sim \Pi(\cdot) \label{priorIII}.
\end{align}
Here, we use $\Pi$ to denote the joint  distribution of $(R,S)$.

We will consider two choices of priors for the sequence $Z_1,Z_2,\ldots$ Our first choice consists of taking independent Gaussian random variables.
If the coefficients $Z_{i}$ are independent with standard deviation $2^{-\ell(i)/2}$, the random draws from this prior are scaled piecewise linear interpolations on a dyadic grid of a Brownian bridge on $[0,1]$ plus the random function $Z_1\psi_1.$ The choice of $\psi_1$ is motivated by the fact that in this case $\var{b(t) \big| S=s, R=\infty} = s^2$ is independent of $t$. 


We construct this second type of prior as follows.   For $\gamma, \sigma^2 >0$, define $V\equiv (V_t,\, t\in [0,1])$ to be the cyclically stationary and centred Ornstein-Uhlenbeck  process. This is a periodic Gaussian process with covariance kernel
\begin{equation}\label{eq:V}
\cov{V(s)}{V(t)} =  \frac{\sigma^2}{2\gamma}\frac{e^{-\gamma h}+e^{-\gamma (1-h)}}{1-e^{-\gamma}}, h = t-s \ge 0.
\end{equation}
This process is cyclically stationary, that is, the covariance only depends on $|t-s|$ and $1-|t-s|$. It is the unique Gaussian and Markovian prior with continuous periodic paths with this property. This makes the cyclically stationary Ornstein-Uhlenbeck prior
an appealing choice which respects the symmetries of the problem.

Each realisation of $V$ is continuous and can be extended to a periodic function on $\RR$.  Then $V$ can be represented as an infinite series expansion in the Faber-Schauder basis:
\begin{equation}\label{eq:Vseries}  V_t= \sum_{i\ge 1} Z_i \psi_i(t) =Z_1 \psi_1(t) + \sum_{j=0}^\infty\sum_{k=1}^{2^j} Z_{j,k} \psi_{j,k}(t) \end{equation}

 Finally by scaling by $S$ and truncating at $R$ we obtain from $V$ the second choice of prior on the drift function $b$. Visualisations of the covariance kernels $\cov{b(s)}{b(t)}$ for first prior (Brownian bridge type) and for the second prior (periodic Ornstein-Uhlenbeck process prior with parameter $\gamma = 1.48$) are shown in fig.~\ref{fig:covariance} (for $S=1$ and $R =\infty$). 

\begin{figure}%
	\begin{center}
		\includegraphics[width=0.45\textwidth]{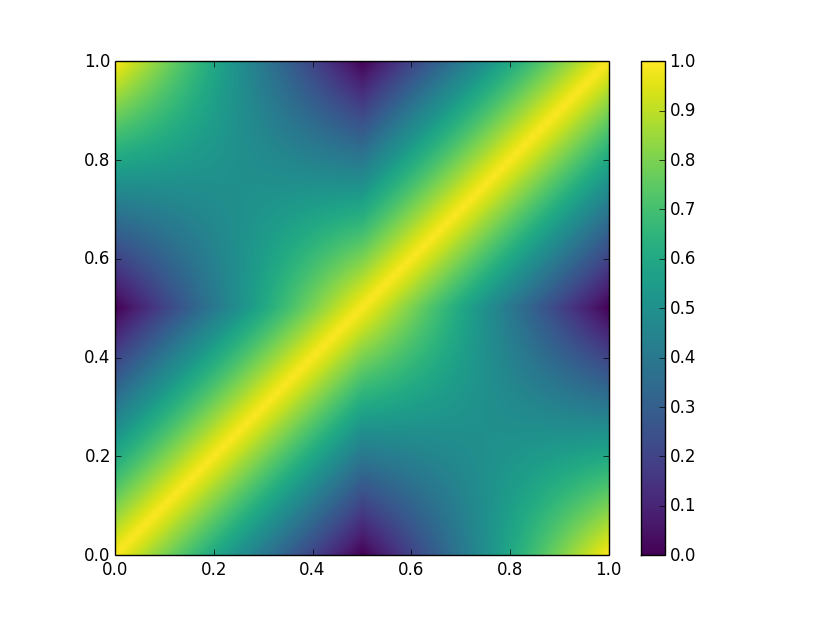}
		\includegraphics[width=0.45\textwidth]{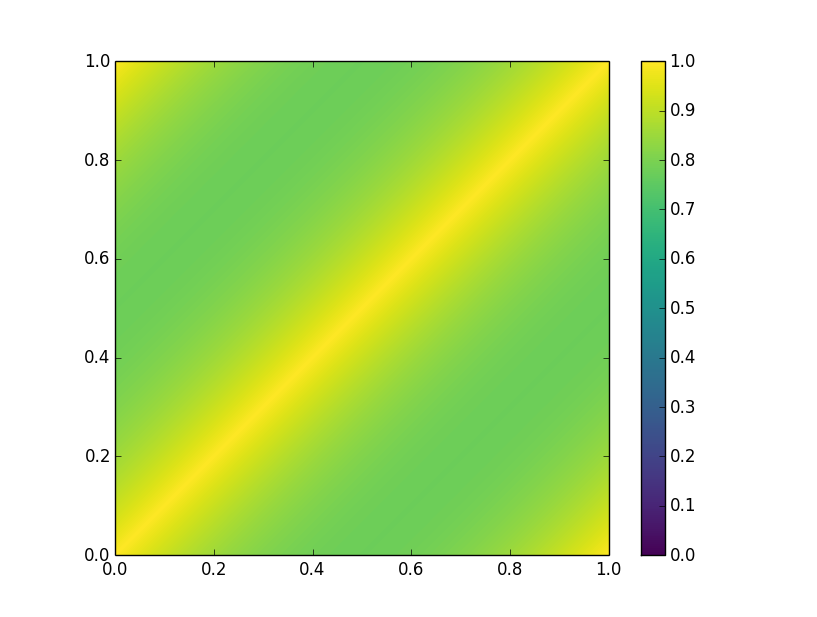}
	\end{center}
	\caption{Heat maps of $(s,t) \mapsto \cov{b(s)}{b(t)}$, in case $S=1$ and $R=\infty$. 
	Left:  Brownian bridge plus the random function $Z_1\psi_1$. Right: periodic Ornstein-Uhlenbeck process with parameter $\gamma=1.48$  and $\sigma^2$ chosen such that $\var{b(s)}=1$.  
		}%
	\label{fig:covariance}%
\end{figure}

\subsection{Sparsity structure induced by choice of $Z_i$}
Conditional on $R$ and $S$, the posterior of $Z^R$ is Gaussian with precision matrix $G^R+\Gamma^R$ (here $G^R$ is the Grammian corresponding to using all basis functions up to and including level $R$).

If the coefficients are independent it is trivial to see that the precision matrix $\Gamma$  does not destroy the sparsity structure of $G$, as defined in \eqref{eq:sigma}. This is convenient for numerical computations.
The next lemma details the situation for periodic Ornstein-Uhlenbeck processes.

\begin{lemma}\label{lem:asymptdiagou}
Let $V$ be defined as in equation \eqref{eq:V} 
\begin{enumerate}
\item The sparsity structure of the precision matrix of the infinite stochastic vector  $Z$ (appearing in the series representation \eqref{eq:Vseries}) equals the  sparsity structure of\/ $G$, as defined in \eqref{eq:sigma}. 
\item 	The entries of the covariance matrix of the random Gaussian coefficients $Z_i$ and $Z_{i'}$, $A_{i,i'} =  \EE Z_i Z_{i'}$, satisfy the following bounds: $A_{11}  = A_{22} = \tfrac{\sigma^2}{2\gamma}\coth(\gamma/2)$ and  for $\gamma  \le 1.5$ and \(i\ge 3\),  
\[		0.95 \cdot 2^{-\ell(i)}\sigma^2/4  \le  A_{ii}  \le 2^{-\ell(i)}\sigma^2/4 \]
and $A_{12} = A_{21} = \tfrac{\sigma^2}{2\gamma}\sinh^{-1}(\gamma/2)$ and for $i\neq i'$ 

\[ 
|A_{ii'}| \le \begin{cases} 
  0.20\sigma^22^{-1.5(\ell(i)\vee\ell( i'))}& \qquad i \wedge i'\le 2<i\vee i',\\
0.37  \sigma^2 2^{-1.5(\ell(i)+\ell(i'))}& \qquad \text{otherwise.} 
\end{cases} \]

\end{enumerate}
\end{lemma}

The proof is given in \cref{sec:proofoflem:asymptdiagou}. By the first part of the lemma, also this prior does  not destroy the sparsity structure of the $G$.
The second part asserts that while the off-diagonal entries of $A^{r}$ are not zero, they are of smaller order than the diagonal entries, quantifying that the covariance matrix of the coefficients in the  Schauder expansion is close to a diagonal matrix.

\section{Posterior contraction for diffusion processes}\label{sec:general-theorem}
The main result in  \cite{vdMeulenvdVaartvZanten} gives sufficient conditions for deriving posterior contraction rates in Brownian semimartingale models.    The following theorem is an adaptation and refinement of Theorem 2.1 and Lemma 2.2 of  \cite{vdMeulenvdVaartvZanten} for diffusions defined on the circle.  
We assume observations $X^{T}$, where  $T \to \infty.$ Let $\Pi^T$ be  a prior on $L^2(\TT)$ (which henceforth may depend on \(T\)) and choose measurable subsets (sieves) $\scr{B}_T \subset L^2(\T)$. Define the balls
\[ B^T(b_0, \eps) = \left\{b\in \scr{B}_T\,:\, \|b_0-b\|_2<\eps\right\}. \]
The {\it $\eps$-covering number} of a set $A$ for a semimetric
$\rho$, denoted by $N(\eps,A,\rho)$, is defined as the minimal
number of $\rho$-balls of radius $\eps$ needed to cover the set
$A$. The logarithm of the covering number is referred to  as the
entropy.

The following theorem characterises the rate of posterior contraction for diffusions on the circle in terms of properties of the prior.

\begin{thm}\label{thm-general}  
Suppose $\{\eps_T\}$ is a sequence of positive numbers such that $T \eps_T^2$ is bounded away from zero. Assume that there is a constant \(\xi>0\) such that for every \(K>0\) there is a measurable set \(\scr B_T\subseteq L^2(\TT)\) and for every \(a>0\) there is a constant \(C>0\) such that for \(T\) big enough
\begin{align}
 \log N\bigl(a\eps_T,B^T(b_0, \eps_T), \|\cdot\|_2\bigr)
&\le C T \eps_T^2,\label{eq:entropycondition}\\
\Pi^T\bigl(B^T(b_0,\eps_T)\bigr) &\ge e^{-\xi T \eps^2_T},\label{eq:smallballprobability}\\
\intertext{and}
\Pi^T\bigl(L^2(\T)\setminus \scr{B}_T\bigr) &\le e^{-KT \eps^2_T}.\label{eq:remainingmasscondition}
\end{align}
Then for every $M_T\to\infty$
\begin{align}P_{b_0} \Pi^T\bigl(b\in L^2(\T): \|b-b_0\|_2\ge M_T \eps_T
\mid X^T\bigr)& \to 0\nonumber\\
\intertext{and for \(K\) big enough,}
	\Pi^T\bigl(L^2(\T)\setminus \scr{B}_T\mid X^T\bigr)&\to 0.\label{eq:posterioroftheremainingmassgoestozero}
\end{align} 
\end{thm}
Equations  \eqref{eq:entropycondition}, \eqref{eq:smallballprobability} and \eqref{eq:remainingmasscondition} are referred to as the entropy condition, small ball condition and remaining mass condition of \cref{thm-general} respectively. The proof of this theorem is in \cref{sec:proofthm-general}.

\section{Theorems on posterior contraction rates}\label{sec:theoremonposteriorcontractionrates}
 
The main result of this section, \cref{thm-main} characterises the frequentist rate of contraction of the posterior probability around a fixed parameter $b_0$ of unknown smoothness using the truncated series prior from  \cref{sec:prior}.

We make the following assumption on the true drift function.
\begin{ass}\label{ass:b0}
The true drift $b_0$ can be expanded in the Faber-Schauder basis, $b_0=z_1\psi_1+\sum_{j=0}^\infty\sum_{k=1}^{2^{j}}z_{jk}\psi_{jk}=\sum_{i\ge1}z_i\psi_i$ and there exists a   \(\beta\in(0,\infty)\) such that 
\begin{equation}\label{eq:betasmoothnessnorm}
\llbracket b_0\rrbracket_{\beta}:=\sup_{i\ge 1} 2^{\beta \ell(i)}|z_{i}|<\infty.
\end{equation}
\end{ass}
Note that we use a slightly different symbol for the norm, as we denote the $L^2$-norm by $\|\cdot\|_2$. 

\begin{rem}
If $\beta \in (0,2)$, then \cref{ass:b0} on $b_0$ is equivalent to assuming $b_0$ to be $B_{\infty,\infty}^\beta$-Besov smooth. It follows from the definition of the basis functions that
\[ z_{jk}=b_0\left((2k-1)2^{-(j+2)}\right)-\frac12 b_0\left(2^{-(j+2)}(2k-2)\right)-\frac12 b_0\left(2^{-(j-2)}2k\right).\]
Therefore it follows from equations (4.72) (with \(r=2\)) and (4.73) (with \(p=\infty\)) in combination with equation (4.79) (with \(q=\infty\)) in \cite{ginenickl2016}, section 4.3, that \(\|b_0\|_\infty+\llbracket b_0\rrbracket_\beta\) is equivalent to the \(B_{\infty,\infty}^\beta\)-norm of $b_0$ for $\beta \in (0,2)$.

If $\beta\in (0,1)$, then  \(\beta\)--H\"older smoothness and \(B^\beta_{\infty,\infty}\)--smoothness coincide (cf.\ Proposition 4.3.23 in \cite{ginenickl2016}).
\end{rem}

For the prior defined in  \cref{priorI,priorII,priorIII} we make the following assumptions.
\begin{ass}\label{ass:prior-cov}
The covariance matrix $A$ satisfies one of the following conditions:
\begin{enumerate}
\item[(A)] For fixed $\alpha>0$, $A_{ii}=2^{-2\alpha \ell(i)}$ and $A_{ii'}=0$ for $i\neq i'$.

\item[(B)] There exists $0 < c_1 < c_2$ and  $0 < c_3$ with $3 c_3 < c_1$ independent from $r$, such that for all  $i, i' \in \scr{I}_r$
\begin{align*} 
&	c_1 2^{-\ell(i)} \le A_{ii} \le  c_2 2^{-\ell(i)},\\
&	|A_{ii'}|  \le c_3  2^{-1.5(\ell(i)+\ell(i'))} \quad \text{ if } i  \ne i'.
\end{align*}
\end{enumerate}
\end{ass}
In particular the second assumption if fulfilled by the prior defined by \cref{eq:V}  if $0 < \gamma \le 3/2$ and any $\sigma^2 > 0$. 
 
\begin{ass}\label{ass:scaling-truncation}
The prior on the truncation level satisfies for some positive  constants \(c_1,c_2\),\begin{equation}\label{eq:definitionofprioronr}
		\begin{split}
	\P(R>r)\le &\exp(-c_12^r r),\\		\P(R=r)\ge &\exp(-c_22^r r).
		\end{split}
\end{equation}
For the prior on the scaling we assume existence of constants $0<p_1<p_2$, $q>0$ and $C>1$ with $p_1>q|\alpha-\beta|$ such that \begin{equation}\label{eq:S}
	\P(S\in[x^{p_1},x^{p_2}])\ge \exp\big(-x^q\big) \quad \text{ for all }x  \ge C. 
\end{equation}
\end{ass}

The prior on \(R\) can be defined as \(R=\lfloor \prescript2{}\log Y\rfloor\), where \(Y\) is   Poisson distributed. \Cref{eq:S} is satisfied for a whole range of distributions, including the popular family of inverse gamma distributions. Since the inverse gamma prior on \(S^2\) decays polynomially (\cref{sec:boundsforinversegammaprior}), condition (A2) of \cite{ghosalshen2015} is not satisfied and hence their posterior contraction results cannot be applied to our prior. We obtain the following result for our prior.

\begin{thm}\label{prop-prior}
Assume $b_0$ satisfies \cref{ass:b0}.
 Suppose the prior satisfies assumptions  \ref{ass:prior-cov} and \ref{ass:scaling-truncation}. Let \(\{\eps_n\}_{n=1}^\infty\) be a sequence of positive numbers that converges to zero. There is a constant \(C_1>0\) such that for any \(C_2>0\) there is a measurable set \(\scr B_n\subseteq L^2(\T)\) such that for every \(a>0\) there is a positive constant \(C_3\) such that for $n$ sufficiently large	
\begin{align}
		\log\P\left(\|b^{R,S}-b_0\|_\infty<\eps_n\right)&\ge -C_1\eps_n^{-1/\beta}|\log\eps_n|\label{eq:smallballprobability0}\\
		\log\P\left(b^{R,S}\notin \scr B_n\right)&\le -C_2\eps_n^{-1/\beta}|\log\eps_n|\label{eq:entropycondition0}\\
		\log N(a\eps, \{b\in \scr B_n\,:\,\|b-b_0\|_2\le \eps_n\},\|\,\cdot\,\|_\infty)&\le C_3\eps_n^{-1/\beta}|\log\eps_n|.\label{eq:remainingmasscondition0}
	\end{align}
\end{thm}

The following theorem is obtained by applying these bounds to  \cref{thm-general} after taking
$\eps_n=(T / \log T)^{-\beta/(1+2\beta)}$.

\begin{thm}\label{thm-main}
Assume $b_0$ satisfies \cref{ass:b0}.
 Suppose the prior satisfies assumptions  \ref{ass:prior-cov} and \ref{ass:scaling-truncation}.
  Then for  all $M_T\to\infty$ 
	\[
	P_{b_0} \Pi^n \left(\left. b\colon \|b - b_0\|_2 \ge M_T\left(\frac T{\log T}\right)^{-\frac{\beta}{1+2\beta}} \;\right|\, X^T\right) \to  0 
	\]
	as $T \to \infty$. 
\end{thm} 

This means that when the true parameter is from \(B_{\infty,\infty}^\beta[0,1],\beta<2\)  a rate is obtained that is optimal possibly up to a log factor. When \(\beta\ge 2\) then \(b_0\) is in particular in the space \(B_{\infty,\infty}^{2-\delta}[0,1],\) for every small positive \(\delta\), and therefore converges with rate essentially \(T^{-2/5}\). 

When 
a different function $\Lambda$ is used, defined on a compact interval of $\re,$ and the basis elements are defined by $\psi_{jk}=\sum_{m\in\ZZ}\Lambda(2^{j}(x-m)+k-1)$; forcing them to be 1-periodic. Then \cref{thm-main} and derived results for applications still holds provided $\|\psi_{jk}\|_\infty = 1$ and
$\psi_{j,k}\cdot\psi_{j,l}\equiv 0$ when $|k-l|\ge d$ for a fixed $d \in \NN$ and the smoothness assumptions on $b_0$ are  changed accordingly. A finite number of basis elements can be added or redefined as long as they are 1-periodic.

It is easy to see that our results imply posterior convergences rates in weaker $L^p$-norms, $1\le p<2,$ with the same rate. When $p\in(2,\infty]$ the $L^p$-norm is stronger than the $L^2$-norm. We apply ideas of \cite{KnapikSalomond2014}  to obtain rates for stronger $L^p$-norms.

\begin{thm}\label{thm:contractionforlpnorms}
 Assume the true drift $b_0$ satisfies assumption \ref{ass:b0}.  Suppose the prior satisfies assumptions  \ref{ass:prior-cov} and \ref{ass:scaling-truncation}.  Let $p\in(2,\infty]$. 
  Then for  all $M_T\to\infty$ 
\[
P_{b_0} \Pi^n \Big(b\colon \|b - b_0\|_p \ge M_TT^{-\frac{\beta-1/2+1/p}{1+2\beta}}(\log T)^{\frac{2\beta -2\beta/p}{1+2\beta}} \Bigm\vert X^T\Big) \to 0
\]
as $T \to \infty.$
\end{thm}

These rates are similar to the rates obtained for the density estimation in \cite{GineNickl}. However our proof is less involved. Note that we have only consistency for \(\beta>1/2-1/p\).

\section{Applications to nonparametric regression and density estimation}\label{sec:auxiliaryresults}

Our general results also apply to other models. The following results are obtained for $b_0$ satisfying \cref{ass:b0} and the prior satisfying assumptions \ref{ass:prior-cov} and \ref{ass:scaling-truncation}.

\subsection{Nonparametric regression model}

As a direct application of the properties of the prior shown in the previous section, we obtain the following result 
for a nonparametric regression problem. Assume
\begin{equation}\label{eq:regression}
X_{i}^n=b_0(i/n)+\eta_i, \quad 0 \le i \le n,	
\end{equation}
with independent Gaussian observation errors $\eta_i\sim \rm N(0,\sigma^2)$.  When we apply \cite{ghosal2007},  example 7.7 to \cref{prop-prior} we obtain, for every \(M_n\to\infty\), 
	\[
	\Pi\left.\left(b\colon \|b - b_0\|_2 \ge M_n\left(\frac n{\log n}\right)^{-\frac{\beta}{1+2\beta}} \;\right|\; X^n\right) \overset{\P^{b_0}}{\longrightarrow} 0 
	\]
		as \(n\to\infty\) and (in a similar way as in \cref{thm:contractionforlpnorms}) for every \(p\in(2,\infty]\),
	\[
	\Pi\left.\left(b\colon \|b - b_0\|_2 \ge M_nn^{-\frac{\beta-1/2+1/p}{1+2\beta}}(\log n)^{\frac{2\beta -2\beta/p}{1+2\beta}} \;\right|\; X^n\right) \overset{\P^{b_0}}{\longrightarrow} 0 
	\]

	as $n \to \infty$.

\subsection{Density estimation}

Let us consider $n$ independent observations $X^n := (X_1,\ldots,X_n)$ with $X_i\sim p_0$ where $p_0$ is an unknown density on $[0,1]$ relative to the Lebesgue measure. Let $\sP$ denote the space of densities on $[0,1]$ relative to the Lebesgue measure.  The natural distance for densities is the Hellinger distance $h$ defined by
\[
h(p,q)^2=\int_0^1\left(\sqrt {p(x)}-\sqrt {q(x)}\right)^2\dd x.
\]

Define the prior on $\sP$ by $p=\frac{\e^b}{\|e^b\|_1},$ where $b$ is endowed with the prior of \cref{thm-main} or its non-periodic version. Assume that \(\log p_0\) is \(\beta\)-smooth in the sense of   \cref{ass:b0}. Applying \cite{ghosal2000}, theorem 2.1 and \cite{vdVaartvZanten}, lemma 3.1 to \cref{prop-prior}, we obtain for a big enough constant \(M>0\)
	\[
	\Pi\left.\left(p\in\sP\colon h(p,p_0)\ge M\left(\frac n{\log n}\right)^{-\frac{\beta}{1+2\beta}}\;\right|\; X^n\right)\xrightarrow{\P_0} 0,
	\]
	as \(n\to\infty\). 
%
%

\section{Proofs}\label{sec:proof}

\subsection{Proof of \cref{lemma:uniquesolutiontosde}}\label{subsec:prooflem1}
Since conditions (ND) and (LI) of \cite[theorem 5.15]{Karatzas-Shreve} hold, the SDE \cref{eq:sde} has a unique weak solution up to an explosion time. 

 Assume without loss of generality that $X_0 = 0$.
Define $\tau_0=0$ and for $i\ge 1$ the random times 
\[\tau_i =  \inf \{t \ge \tau_{i-1} \colon |X_{t} - X_{\tau_{i-1}}| = 1\}.\]
By periodicity of drift and the Markov property the random variables $U_i = \tau_{i}-\tau_{i-1}$ are independent and identically distributed. 

Note that 
\[
\inf \{t\colon X_t = \pm n\} \ge \sum_{i=1}^n U_i
\]
and hence non-explosion follows from $\lim_{n\to\infty}\sum_{i=1}^n U_i=\infty$ almost surely. The latter holds true since $U_1>0$ with positive probability, which is clear from the continuity of diffusion paths.

\subsection{Proof of \cref{lem:asymptdiagou}}\label{sec:proofoflem:asymptdiagou}

\emph{Proof of the first part.}
For the proof we introduce some notation: for any $(j, k)$, $(j', k')$ we write
$(j, k) \prec (j', k')$  if $\supp\, \psi_{j',k'}\subset \supp\, \psi_{j,k}$. The set of indices become a lattice with partial order $\prec$, and by  $(j,k) \vee (j',k')$ we denote the supremum.  
Identify $i$ with $(j,k)$ and similarly $i'$ with $(j',k')$. 

For $i>1$, denote by  $t_i$ the time points in $[0,1]$ corresponding to the maxima of $\psi_i$. Without loss of generality assume $t_i<t_{i'}$.
 We have  $G_{i,i'}=0$ if and only if  the interiors of the supports of  $\psi_i$ and $\psi_{i'}$ are disjoint. 
In that case    \begin{equation}\label{eq:suppineq}	\max \supp\,\psi_{j,k} \le t_{(j,k) \vee (j',k')} \le \min \supp\, \psi_{j',k'}.\end{equation}

 The values of $Z_i$ can be found by the midpoint displacement technique.
The coefficients are given by $Z_1 = V_0$, $Z_2 = V_{\frac12}$ and for $j\ge 1$
 \[	Z_{j,k} = V_{2^{-j}(k-1/2)} - \frac12\left( V_{2^{-j}(k-1)} + V_{{2^{-j}k}}\right).\]
 As $V$ is a Gaussian process, the vector $Z$ is mean-zero Gaussian, say with (infinite) precision matrix $\Gamma$. 
Now $\Gamma_{i,i'}=0$ if there exists a set $\scr{L}\subset \NN$ such that $\scr{L} \cap \{i,i'\}=\varnothing$ for which conditional on $\{ Z_{i^\star},\, i^\star \in \scr{L}\}$, $Z_i$ are $Z_{i'}$ are independent. 

 Define $(j^\star, k^\star)=(j,k) \vee (j',k')$ and \[\scr{L}=\{i^\star \in \NN\,:\, i^\star=2^j+k, \text{ with } j\le j^\star\}.\] 

The set $\{ Z_{i^\star},\, i^\star \in \scr{L}\}$  determine the process $V$ at all times $k 2^{-j^\star-1}$, $k=0\ldots,2^{j^\star+1}$. 

Now $Z_i$ and $Z_{i'}$ are conditionally independent given $\{V_t, t=k 2^{-j^\star-1},\, k=0\ldots,2^{j^\star+1}\}$ by \eqref{eq:suppineq} and the Markov property of the nonperiodic Ornstein-Uhlenbeck process.

The result follows since $\sigma(\{Z_{i^\star},\, i^\star\in \scr{L}\})=\sigma(\{V_t, t=k 2^{-j^\star-1},\, k=0\ldots,2^{j^\star+1}\})$.

\begin{lemma}\label{lemma:factor}
Let $K(s,t)=\EE {V}_s {V}_t= \frac{\sigma^2}{2\gamma}\frac{1}{1-e^{-\gamma}} ( e^{-\ga |t-s|}+e^{-\ga (1-|t-s|)})$. 
If $x \notin (s,t) $
\[
\frac12K(s, x) - K(\tfrac{s+t}2,x) +\frac12 K(t,x) = 2 \sinh^2(\gamma\tfrac{t-s}{4}) K(\tfrac{t+s}{2}, x)
\]
\end{lemma}
\begin{proof}
Without loss of generality assume that $ t \le x \le 1$.
With $m= (t+s)/2$ and $\delta = (t-s)/2$
\[
( e^{-\ga |s-x|}+e^{-\ga (1-|s-x|)}) - 2( e^{-\ga |m-x|}+e^{-\ga (1-|m-x|)}) + ( e^{-\ga |t-x|}+e^{-\ga (1-|t-x|)}) 
\]
\[
=  e^{-\ga |t-x|}e^{-2\gamma\delta} - 2 e^{-\ga |t-x|} e^{-\gamma\delta}+ e^{-\ga |t-x|} +
e^{-\ga (1-|s-x|)} - 2e^{-\ga (1-|s-x|)}e^{-\gamma\delta}+e^{-\ga (1-|s-x|)}e^{-2\gamma\delta}
\]
\[
=  (1-e^{-\gamma\delta})^2 (e^{-\ga |t-x|} + e^{-\ga (1-|s-x|)}) =  (1-e^{-\gamma\delta})^2 e^{\gamma\delta}(e^{-\ga |m-x|} + e^{-\ga (1-|m-x|)})
\]
The result follows from $(1-e^{-\gamma\delta})^2 e^{\gamma\delta}= 4\sinh^2(\gamma\delta/2)$ and scaling both sides with $\tfrac12 \frac{\sigma^2}{2\gamma}\frac{1}{1-e^{-\gamma}} $.
\end{proof}

\emph{Proof of the second part.}
	Denote by $[a, b]$, $[c, d]$ the support of $\psi_i$ and $\psi_{i'}$ respectively and let $m = (b+a)/2$ and $n = (d+c)/2$ but for \(i=1\), let \(m=0\). 
$Z_1 = V(0)$, $Z_2 = V_{1/2}$	and $\var{Z_1} = \var{Z_2} =  \frac{\sigma^2}{2\gamma}\coth(\gamma/2)$, and
$\cov{Z_1}{Z_2} = \frac{\sigma^2}{2\gamma}\sinh^{-1}(\gamma/2)$.
Note that the $2\times2$ covariance matrix of $Z_1$ and $Z_2$ has eigenvalues $\tfrac{\sigma^2}{2\gamma}  \operatorname{tanh}(\gamma/4)$ and 
$\tfrac{\sigma^2}{2\gamma} \coth(\gamma/4)$ and is strictly positive definite.
	
By midpoint displacement, $2Z_{i} = 2V_{m} - V_{a} - V_{b}$, $i > 2$ and $K(s,t)=\EE {V}_s {V}_t= \frac{\sigma^2}{2\gamma}\frac{1}{1-e^{-\gamma}} ( e^{-\ga |t-s|}+e^{-\ga (1-|t-s|)})$. 

Assume without loss of generality $b-a \ge d-c$. Define $\delta$ to be the halfwidth of the smaller interval, so that $\delta := (d-c)/2= 2^{-j'-1}$. Then \[
(b-a)/2  =2^{-j-1}=h \delta,\quad \text{with}\quad h=2^{j'-j}. \]  
Consider three cases:
\begin{enumerate}
\item The entries on diagonal, $i = i'$;
\item  The interiors of the supports of \(\psi_i\) and \(\psi_{i'}\) are non-overlapping;
\item The support of \(\psi_{i'}\) is contained in the support of \(\psi_i\). 
\end{enumerate}

{\it Case 1.} 
By elementary computations for $i > 2$,
\[
4 \frac{2\ga}{\si^2} (1-e^{-\gamma}) A_{ii}=  6(1+e^{-\gamma}) + 2(e^{-\gamma 2 \delta} + e^{-\gamma(1- 2\delta)}) -  8 ( e^{-\gamma\delta} +  e^{-\gamma(1-\delta)} )= 
\]
\[
= 2   (1-e^{-\gamma\delta}) ( 3 -e^{-\gamma\delta}) +  2e^{-\gamma}  (1-e^{\gamma\delta}) ( 3 -e^{\gamma\delta}) .
\]
As $\delta \le \tfrac14$ and under the assumption $\gamma \le 3/2$ the last display can be  bounded by 
\[
0.9715\cdot 4 \gamma\delta  (1-e^{-\gamma})   \le 4 \frac{2\ga}{\si^2} (1-e^{-\gamma}) A_{ii} \le 4 \gamma\delta  (1-e^{-\gamma})  .
\]

Hence \(
 0.9715\cdot 2^{-j}\sigma^2/4\le 	A_{ii}\le 2^{-j}\sigma^2/4\).

{\it Case 2.} 
Necessarily $i, i' > 2$. By twofold application of lemma \ref{lemma:factor}
\begin{equation}\label{eq:calculationofAii'}
\begin{split} A_{ij}  &= 
(K(c,b)- 2K(n,b)+K(d,b))/4\\
&\quad -2(K(c,m) - 2K(n,m)  + K(d,m))/4\\
&\quad +(K(c,a)- 2K(n,a)+K(d,a))/4\\
&=2 \sinh^2(\gamma\tfrac{d-c}{4}) (K(n,b) - 2K(n,m) + K(n,a))/2\\
&=4 \sinh^2(\gamma\tfrac{b-a}{4})\sinh^2(\gamma\tfrac{d-c}{4}) K(n, m).
\end{split}
\end{equation}

Using the convexity of \(\sinh\) we obtain the bound 
\begin{equation}\label{eq:boundforthesinh}
2\sinh^2(x/2) \le 0.55 x^2
\end{equation}

 for $0 \le x \le 1$. Note that \(f(x)=e^{-\gamma x}+e^{-\gamma(1-x)}\) is convex on \([0,1]\), from which we derive \(f(x)\le 1+e^{-\gamma}\). Using this bound, and the fact that for \(\gamma\le3/2\), 
\begin{equation}\label{eq:upperboundforKmn}
 \gamma^2 K(n,m)\le \tfrac{\sigma^2}2\gamma\coth(\gamma/2) \le \sigma^2( 1 +\gamma/2),
\end{equation} 
  which can be easily seen from a plot, that
\begin{align*}
| A_{ii'}| \le & 0.55^2\gamma^4\cdot  2^{-2j-2} \cdot 2^{-2j'-2}|K(n,m)|\\
\le &0.0095\sigma^2\gamma^2(1+\gamma/2)2^{-1.5(j + j')}.
\end{align*}

{\it Case 3.}

For $ i' > 2$, $i = 1$ with $m = 0$ or $i = 2$ with $m = \frac12$, using \cref{eq:upperboundforKmn}, we obtain
\begin{equation}\label{eq:A1}
\begin{split}
&|A_{ii'}|  = |K(m,n) 
- \frac12K(m,c) - \frac12K(m,d)|\\
 \le & 2 \sinh^2(\gamma\tfrac{d-c}{4}) K(m,n) \\
 \le & 0.55\gamma^22^{-2j'-2}K(m,n)\\
 \le & 0.098\sigma^2(1+\gamma/2) 2^{-1.5j}. 
\end{split}
\end{equation}

When \(i,i'>2\) then, using the calculation \cref{eq:calculationofAii'} and \cref{lemma:factor} noting that \(a,b\) and \(m\) are not in \((c,d)\), we obtain
\[ A_{ii'}  = 2 \sinh^2(\gamma\tfrac{d-c}{4}) (K(n,b) - 2K(n,m) + K(n,a))/2.
\]

Write \(x=\gamma|a-m|=\gamma|b-m|=\gamma h\delta\) and \(\alpha=\frac{|m-n|}{|b-m|}\in(0,1)\). A simple computation then shows 
\begin{align*}
	&e^{-\gamma  |b-n|} - 2e^{-\gamma|m-n|} + e^{-\gamma|a-n|}\\
	=&e^{-(1+\alpha)x}-2e^{-\alpha x}+e^{-(1-\alpha)x}.
\end{align*}
The derivative of \(f(\alpha):=e^{-(1+\alpha)x}-2e^{-\alpha x}+e^{-(1-\alpha)x}\) is nonnegative, for \(\alpha,x>0\) hence \(f(\alpha)\) is increasing and so \(f(0)\le f(\alpha)\le f(1)\). Note that \(
	f(0)= 2e^{-x}-2\ge -2x, \text{ for }x>0\) and \(f(1)= e^{-2x}-2e^{-x}+1=:g(x)\). Maximising \(g'(x)\) over  \(x>0\) gives \(g'(x)\le 1/2\) and \(g(0)=0\) and therefore \(f(1)=g(x)\le x/2\). 

It follows that
\begin{align*}
-2\gamma h\delta \le 	e^{-\gamma  |b-n|} - 2e^{-\gamma|m-n|} + e^{-\gamma|a-n|}\le \gamma h\delta / 2.
\end{align*}

For the other terms we derive the following bounds. Write 
\begin{align*}
	&e^{-\gamma (1-|b-n|)} - 2e^{-\gamma (1-|m-n|)} + e^{-\gamma(1-|a-n|)}\\
	= & e^{-\gamma+(1+\alpha)x}-2e^{-\gamma+\alpha x}+e^{-\gamma +(1-\alpha)x}=:h(\alpha).
\end{align*}

Now \(h(\alpha)\) is decreasing for \(x\le \log 2\) and convex and positive for \(x\ge \log2\). In both case we can bound \(h(\alpha)\) by its value at the endpoints \(\alpha=0\) and \(\alpha=1\). Using that \(2x\le \gamma\) we obtain \(0\le h(0)= e^{-\gamma}(2e^x-2)\le 2x\) and \(	0 \le h(1)= e^{-\gamma}\big(e^{2x}-2e^x+1\big)\le 2x\). So \(0\le h(\alpha)\le 2\gamma h\delta\).

Using the bound \cref{eq:boundforthesinh} and  $x/(1-\exp(-x))\le (1+x)$ we obtain \[
|A_{ii'}|\le  0.061  \sigma^2 \gamma (1+\gamma) 2^{-1.5(j+j')}.
\]

\subsection{Proof of \cref{thm-general}}\label{sec:proofthm-general}

A general result for deriving contraction rates for Brownian semi-martingale models was proved in \cite{vdMeulenvdVaartvZanten}. Theorem \ref{thm-general} follows upon verifying the assumptions of this result for the diffusion on the circle. These assumptions are easily seen to boil down to: \begin{enumerate}
  \item For every \(T>0\) and \(b_1,b_2\in L^2(\TT)\) the measures \(P^{b_1,T}\) and \(P^{b_2,T}\) are equivalent.	
  \item The posterior as defined in equation \cref{eq:posterior} is well defined. 
  \item Define the (random) {\it Hellinger semimetric} $h_T$ on  $L^2(\TT)$  by 
\begin{equation}\label{hellinger}
h_T^2(b_1,b_2):= \int_0^{T}
\Bigl(b_1-b_2\Bigr)^2(X_t)\,\dd t,
\quad b_1,\, b_2 \in L^2(\TT). 
\end{equation}
There are constants \(0<c<C\) for which 
\[ \lim_{T\to\infty}
P^{\th_0,T}\Bigl(c\sqrt T\|b_1-b_2\|_2\le h_T(b_1,b_2) \le C\,\sqrt T\|b_1-b_2\|_2, \forall\,
,b_1, b_2\in L^2(\TT) \Bigr) =1. 
\]
\end{enumerate}

We start by verifying the third condition. Recall that the local time of the process \(X^T\) is defined as the random process \(L_T(x)\) which satisfies \begin{align*}
	\int_0^T f(X_t)\dd t=\int_\re f(x)L_T(x)\dd x.
\end{align*}
For every measurable function \(f\) for which the above integrals are defined. 
 Since we are working with 1-periodic functions, we define the periodic local time by
\[
	\mathring L_T(x)=\sum_{k\in Z}L_T(x+k).
\]

Note that \(t\mapsto X_t\) is continuous with probability one. Hence the support of \(t\mapsto X_t\) is compact with probability one.  Since \(x \mapsto L_T(x)\) is only positive on the support of \(t\mapsto X_t\), it follows that  the sum in the definition of \(\mathring L_T(x)\) has only finitely many nonzero terms and is therefore well defined. For a one-periodic function \(f\) we have \begin{align*}
	\int_0^Tf(X_t)\dd t=\int_0^1f(x)\mathring L_T(x)\dd x,
\end{align*}
provided the involved integrals exists. It follows from  \cite[Theorem 5.3]{SchauervZanten} that \(
	\mathring L_T(x)/T\)
converges to a positive  deterministic function only depending only on \(b_0\) and which is bounded away from zero and infinity. Since  the Hellinger distance can be written as 
\[
	h_T(b_1,b_2)
	=  \sqrt T\sqrt{\int_0^1(b_1(x)-b_2(x))^2\frac{\mathring L_T(x)}T\dd t}
\]
it follows that the third assumption is satisfied with  \(d_T(b_1,b_2)=\sqrt T\|b_1-b_2\|_2\). 

Conditions 1 and 2 now follow  by arguing precisely as in lemmas A.2 and  3.1 of \cite{waaijzanten2015} respectively  (the key observation being that  the convergence result of \(\mathring L_T(x)/T\) also holds when \(\int_0^1b(x)\dd x\) is nonzero, which is assumed in that paper).

The stated result follows from  Theorem 2.1 in \cite{vdMeulenvdVaartvZanten} (taking $\mu_T=\sqrt{T} \eps_T$ in their paper).

\subsection{Proof of \cref{prop-prior} with \cref{ass:prior-cov} (A)}
The proof proceeds by verifying the conditions of theorem \ref{thm-general}.
By \cref{ass:b0} the true drift can be represented as $
 b_0=z_1\psi_1+\sum_{j=0}^\infty\sum_{k=1}^{2^{j}}z_{jk}\psi_{jk}$. For $r\ge 0$, define its truncated version by \[ b^r_0=z_1\psi_1+\sum_{j=0}^r\sum_{k=1}^{2^{j}}z_{jk}\psi_{jk}. \]

\subsubsection{Small ball probability}\label{subsec:smallball}

For $\eps >0 $ choose an integer $r_\eps$ with
\begin{equation}\label{eq:conditiononrstar}
C_\beta \eps^{-1/\beta} \le 2^{r_\eps} \le 2C_\beta \eps^{-1/\beta} \qtext{ where }
C_\beta= \frac{\llbracket b_0\rrbracket_\beta^{1/\beta}}{(2^\beta-1)^{1/\beta}}.
\end{equation}
For notational convenience we will write $r$ instead of $r_\eps$ in the remainder of the proof. 
By \cref{lem:approximationerror}
we have \(\|b_0^{r}-b_0\|_\infty \le\eps\). Therefore
\[  \|b^{r,s}-b_0\|_2 \le \|b^{r,s}-b_0^{r}\|_2 + \|b^{r}_0-b_0\|_2 
\le \|b^{r,s} - b_0^{r}\|_\infty + \eps\]
which implies
\[ \P\left(\|b^{r,s}-b_0\|_2 < 2\eps\right) \ge \P\left( \|b^{r,s}-b_0^{r}\|_\infty <\eps\right). \]
Let  $f_S$ denotes the probability density of $S$. 
For any $x > 0$, we have
\begin{align}
\P\left(\|b^{R,S}-b_0\|_2 < 2\eps\right) &= \sum_{r\ge 1} \P(R=r) \int_0^\infty \P\left(\|b^{r,s}-b_0\|_2 < 2\eps\right) f_S(s) \,\dd s\nonumber \\ & \ge \P(R=r) \inf_{s\in [L_\eps, U_\eps]}  \P\left(\|b^{r,s}-b_0^{r}\|_\infty < \eps\right) \int_{L_\eps}^{U_\eps} f_S(s) \dd s, \label{eq:lowerboundofsmallballinthreeparts}
\end{align}
where 
\[ L_\eps = \eps^{-\frac{p_1}{q\beta}} \qquad \text{and} \qquad U_\eps = \eps^{-\frac{p_2}{q\beta}} \]
and $p_1, p_2$ and $q$ are taken from \cref{ass:scaling-truncation}. 
For $\eps$ sufficiently small, we  have by the second part of  \cref{ass:scaling-truncation} 
\[
\int_{L_\eps}^{U_\eps} f_S(s) \dd s \ge \exp\big(-\eps^{-\frac1\beta}\big)
\]
By  choice of \(r\) and the first part of \cref{ass:scaling-truncation},  there exists a positive constant $C$ such that 
\[\P(R=r)\ge\exp\Big(-c_22^{r}{r}\Big)\ge  \exp\Big(-C\eps^{-\frac1\beta}|\log\eps|\Big),\] for $\eps$ sufficiently  small.  

For lower bounding the middle term in equation \eqref{eq:lowerboundofsmallballinthreeparts}, we write
\[ b^{r,s}-b_0^{r} = (sZ_1-z_1)\psi_1+\sum_{j=0}^{r} \sum_{k=1}^{2^{j}} (s Z_{jk}-z_{jk}) \psi_{jk} \]
which implies 
\[ \|b^{r,s}-b_0^{r}\|_\infty \le|sZ_1-z_1|+ \sum_{j=0}^{r} \max_{1\le k \le 2^{j}} |s Z_{jk}-z_{jk}| \le (r+2) \max_{i \in \scr{I}_{r}} |s Z_{i}-z_{i}| .\]
This gives the bound
\[ \P\left(\|b^{r,s}-b_0^{r}\|_\infty <\eps\right) \ge  \prod_{i \in \scr{I}_{r}}  \P\Big(|sZ_{i}-z_{i}| <\frac\eps{ r+2}\Big). \]
By choice of the \(Z_i,\) we have for all \(i\in\{1,2,\ldots\}, 2^{\alpha\ell(i)}Z_i\) is standard normally distributed and hence
\begin{align*}
\log \P\left(|sZ_{i}-z_{i}| <\frac\eps{r+2}\right) & = \log \P\left(\left| 2^{\alpha \ell(i)} Z_{i} - 2^{\alpha \ell(i)} z_{i}/s\right| < \frac{2^{\alpha \ell(i)} \eps}{(r+2) s}\right) \\ 
&\ge \log \left( \frac{2^{\alpha \ell(i)}\eps}{ (r+2) s}\right) - \frac{2^{2\alpha \ell(i)}\eps^2}{(r+2)^2 s^2} - \frac{2^{2\alpha \ell(i)}z_{i}^2}{s^2} +\tfrac12\log\left( \tfrac{2}{\pi}\right),
\end{align*}
where the inequality follows from lemma \ref{lem:normalball}. The third term can be further bounded as we have
\[ 2^{2\alpha   \ell(i)} z_{i}^2 = 2^{2(\alpha-\beta) \ell(i)} 2^{2\beta \ell(i)} z_{i}^2 \le 2^{2(\alpha-\beta) \ell(i)} \llbracket b_0\rrbracket_\beta^2. \]
Hence 
\begin{equation}\label{eq:boundforsZiminuszi}
 \log P\left(|sZ_{i}-z_{i}| <\frac\eps{r+2}\right) \ge  \log \left( \frac{2^{\alpha \ell(i)}\eps}{ (r+2) s}\right) - \frac{2^{2\alpha \ell(i)}\eps^2}{(r+2)^2 s^2} - \frac{2^{2(\alpha-\beta) \ell(i)} \llbracket b_0\rrbracket_\beta^2}{s^2} +\tfrac12\log\left( \tfrac{2}{\pi}\right). 
 \end{equation}
For $s \in [L_\eps, U_\eps]$ and \(i\in\scr{I}_{r}\) we will now derive bounds on the first three terms on the right of \cref{eq:boundforsZiminuszi}. For $\eps$ sufficiently small we have $r\le r+2\le 2r$ and then  inequality \eqref{eq:conditiononrstar} implies 
\[ \log C_\beta \le r+2\le 2\log(4C_\beta) +\frac2{\beta} |\log \eps|. \]

{\it Bounding the first term on the RHS of \eqref{eq:boundforsZiminuszi}.} For $\eps$ sufficiently small, we have
\begin{align*}
\log \left(\frac{(r+2) s}{2^{\alpha \ell(i)} \eps}\right) &\le \log \left(\frac{(r+2) U_\eps}{\eps}\right) =\log \left((r+2) \eps^{-\left(1+\frac{p_2}{q\beta}\right)}\right)\\ & \le \log\left\{ 2 \log(4C_\beta) +\frac2{\beta} |\log \eps| \right\} + \left(1+\frac{p_2}{q\beta}\right) |\log \eps|  \\&\le \tilde{C}_{p_2, q, \beta} |\log \eps|,
\end{align*}
where $\tilde{C}_{p_2, q, \beta}$ is a positive constant. 

{\it Bounding the second term on the RHS of \eqref{eq:boundforsZiminuszi}.} For $\eps$ sufficiently small, we have
\[
\frac{2^{2\alpha \ell(i)} \eps^2}{(r+2)^2 s^2} \le \frac{2^{2\alpha r} \eps^2}{(\log C_\beta)^2 L_\eps^2} \le \frac{(2C_\beta)^{2\alpha}}{(\log C_\beta)^2} \eps^{\frac2{\beta}\left(-\alpha +\beta +p_1/q\right)} \le 1.
\]
The final inequality is immediate in case $\alpha=\beta$, else if suffices to verify that the exponent is non-negative under the assumption $p_1>q|\alpha-\beta|$. 

{\it Bounding the third term on the RHS of \eqref{eq:boundforsZiminuszi}.} For $\eps$ sufficiently small, in case $\beta\ge \alpha$ we have
\[ \frac{2^{2(\alpha-\beta) \ell(i)} \llbracket b_0\rrbracket_\beta^2}{s^2} \le  \llbracket b_0\rrbracket_\beta^2 L_\eps^{-2}\le 1. \]
In case $\beta<\alpha$ we have 
\[  \frac{2^{2(\alpha-\beta) \ell(i)} \llbracket b_0\rrbracket_\beta^2}{s^2} \le \frac{2^{2(\alpha-\beta) r} \llbracket b_0\rrbracket_\beta^2}{L_\eps^2} 
\le (2C_\beta)^{2(\alpha-\beta)} \eps^{\frac2{\beta}\left(p_1/q-\alpha+\beta\right)} \le 1\]
as the exponent of $\eps$ is positive 
 under the assumption $p_1>q|\alpha-\beta|$.

Hence for \(\eps\) small enough, we have 
\[  \log P\left(|sZ_{i}-z_{i}| <\frac\eps{r+2}\right) \ge -\tilde{C}_{p_2,q, \beta} |\log \eps| -3.\]
As $-2^{r+1}\ge -4C_\beta \eps^{-1/\beta}$ we get
\begin{align*}
	\log \inf_{s\in[x^{p_1}, x^{p_2}]}P\left(\|b^{r,s}-b_0^{r}\|_\infty <\eps\right)&\ge -4C_\beta \eps^{-1/\beta} \left(  \tilde{C}_{p_2,q, \beta} |\log \eps| +3\right)
\\	&\gtrsim -\eps^{-1/\beta}|\log\eps|. 
\end{align*}
We conclude that the right hand side of \cref{eq:lowerboundofsmallballinthreeparts} is bounded below by $\exp\big({-C_1}\eps^{-1/\beta}|\log\eps|\big)$, for some positive constant \(C_1\) and sufficiently small \(\eps\).

\subsubsection{Entropy and remaining mass conditions}\label{subsec:definitionsieves}

For $r\in \{0,1,\ldots\}$ denote by $\scr C_r$ the linear space spanned by \(\psi_1\) and $\psi_{jk}$, $0 \le j \le r,$ $k \in 1, \dots, 2^j$, and define
\[	\scr C_{r,t} := \left\{b \in \scr C_r, \llbracket b \rrbracket_\alpha \le t\right\}.\]

\begin{prop}\label{prop:entropyandremainingmass}
	For any $\eps>0$ 
\[    	\log N(\eps,\scr C_{r,t},\|\,\cdot\,\|_\infty)\le 2^{r+1}\log(3A_\alpha t\eps^{-1}),\]
where $A_\alpha=\sum_{k=0}^\infty 2^{-k\alpha}$.
\end{prop}

\begin{proof} 
 
We follow \cite[\S 3.2.2]{vdMeulenvdVaartvZanten}. Choose $\eps_0, \ldots, \eps_r>0$ such that $\sum_{j=0}^r \eps_j \le \eps$. Define 
\[ U_j =\begin{cases} \left[-2^{-\alpha j} t, 2^{-\alpha j} t\right]^{2^j} & \quad \text{if $j\in \{1,\ldots, r\}$} \\
[-t, t]^2  & \quad \text{if $j=0$}\end{cases}. \]
For each $j\in \{1,\ldots, r\}$, let $E_j$ be a minimal $\eps_j$-net with respect to the max-distance on $\RR^{2^j}$ and let $E_0$ be a minimal $\eps_0$-net with respect to the max-distance on $\RR^2$. Hence, if $x\in U_j$, then there exists a $e_j \in E_j$ such that $\max_k |x_k-e_k| \le \eps_j$. 

Take $b\in \scr{C}_{r,t}$ arbitrary: $b=z_1\psi_1 + \sum_{j=0}^r \sum_{k=1}^{2^j} z_{jk}\psi_{jk}$. Let $\tilde{b} = 
e_1\psi_1 + \sum_{j=0}^r \sum_{k=1}^{2^j} e_{jk}\psi_{jk}$, where $(e_1, e_{0,1})\in E_0$ and $(e_{j1},\ldots, e_{j2^j}) \in E_j$ (for $j=1,\ldots, 2^j$). We have
\begin{align*}
\|b-\tilde{b}\|_\infty & \le |z_1-e_1| \|\psi_1\|_\infty + \sum_{j=0}^r \max_{1\le k\le 2^j} |z_{jk}-e_{jk}| \|\psi_{jk}\|_\infty \\ & \le |z_1-e_1| + \sum_{j=0}^r  \max_{1\le k\le 2^j} 2^{j\alpha} |2^{-j\alpha} z_{jk}-2^{-j\alpha}  e_{jk}| . 
\end{align*}
This can be bounded by $\sum_{j=0}^r \eps_j$ by an appropriate choice of the coefficients in $\tilde{b}$. In that case we obtain that $\|b-\tilde b\|_\infty \le \eps$. This implies
\[ \log N(\eps,\scr C_{r,t},\|\,\cdot\,\|_\infty)\le \sum_{j=0}^r \log |E_j| \le \sum_{j=0}^r 2^j \log \left(\frac{3\cdot 2^{-\alpha j} t}{\eps_j}\right). \]
The asserted bound now follows upon choosing  $\eps_j =\eps 2^{-j\alpha} /A_\alpha$.

\end{proof}

\begin{prop}\label{prop:entropybound-h}
There exists a constant a positive constant $K$ such that 
\[ \log  N\left(a\eps,  \left\{b\in \scr{C}_r\,:\, \|b-b_0\|_2\le \eps\right\} , \|\cdot\|_2\right)\le 2^{r+1} \log\left(6 A_\alpha K 2^{\alpha r}\right). \]
\end{prop} 
\begin{proof}
There exists a positive $K$ such that 
\[ \left\{b\in \scr{C}_r\,:\, \|b-b_0\|_2\le a\eps\right\} \subset  \left\{b\in \scr{C}_r\,:\, \|b\|_2\le K \right\}. \]
By \cref{lem:modulusofcontinuityupperbound}, this set is included in the set
\begin{equation}\label{eq:set}   \left\{b\in \scr{C}_r\,:\, \|b\|_\infty \le  \sqrt{3} 2^{(r+1)/2} K\right\}. \end{equation}
By  \cref{lem:boundforcoefficientsintermsofthelinfinitynorm}, for any $b=z_1\psi_1 +\sum_{j=0}^r \sum_{k=1}^{2^j} z_{jk}\psi_{jk}$ in this set we have \[\max\left\{|z_1|, |z_{jk}|, j=0,\ldots, r,\: k=1\ldots,2^j\right\} \le 2 \|b\|_\infty  \sqrt{3} 2^{(r+1)/2} K.\]
 Hence, the set \cref{eq:set} is included in the set $\left\{b\in \scr{C}_r\,:\, \llbracket b\rrbracket_\alpha \le a(r,\eps)\right\}=\scr{C}_{r, a(r,\eps)}$, where $a(r,\eps)=2^{1+\alpha r} \sqrt{3} 2^{(r+1)/2} K$. 

Hence,
\[ N\left(a\eps,  \left\{b\in \scr{C}_r\,:\, \|b-b_0\|_2\le \eps\right\} , \|\cdot\|_2\right) \le N\left(\eps, \scr{C}_{r, a(r,\eps)}, \|\cdot\|_2\right). \]
Using \cref{lem:modulusofcontinuityupperbound} again the latter can be bounded by 
\[ N\left(\eps\sqrt{3}2^{(r+1)/2}, \scr{C}_{r, a(r,\eps)}, \|\cdot\|_\infty\right)\]
The result follows upon applying \cref{prop:entropyandremainingmass}.
\end{proof}

We can now finish the proof for the entropy and remaining mass  conditions. 
Choose \(r_n\) to be the smallest integer so that  $2^{r_n}\ge L\eps_n^{-\frac1\beta}$, where \(L\) is a constant, and set $ \scr B_n=\scr C_{r_n}$. The entropy bound then follows directly from \cref{prop:entropybound-h}. 

For the remaining mass condition, using   \cref{ass:scaling-truncation}, we obtain
\[
	\P\left(b^{R,S}\notin \scr B_n\right)=\P(R>r_n)
	\le \exp\big(-c_12^{r_n}r_n\big)
\le \exp\big(-C_3\eps_n^{-\frac1\beta}|\log\eps_n|\big),
\]
 and note that the constant \(C_3\) can be made arbitrarily big by choosing \(L\) big enough.

\subsection{Proof of \cref{prop-prior} under  \cref{ass:prior-cov} (B)}

We start with a lemma.

\begin{lemma}\label{lem:equivalentmatrices}

Assume there exists $\,0 < c_1 < c_2$ and  $\,0 < c_3$ with $c_3 < c_1$ independent from $r$, such that for all  $i, i', 2\le \ell(i),\ell(i')\le r$, 
\begin{align} \label{bound}
&	c_1 2^{-\ell(i)} \le A_{ii} \le  c_2 2^{-\ell(i)},\\
&	|A_{ii'}|  \le c_3  2^{-1.5(\ell(i)+\ell(i'))} \quad \text{ if } i  \ne i'.
\end{align}
Let \(\widetilde A=(A_{ii'})_{2\le \ell(i),\ell(i')\le r}\) (so the right-lower submatrix of \(A^r\)). Then for all  $x \in \RR^{|\scr{I}_r|-2}$ 
\[
(c_1 -c_3)  x'\widetilde \Lambda x \le  x' \widetilde A  x \le 2 c_2 x' \widetilde \Lambda x  .
\]
where \(\widetilde \Lambda=(\widetilde \Lambda_{ii'})_{2\le \ell(i),\ell(i')\le r}\) is the diagonal matrix with  $\widetilde\Lambda_{ii} = 2^{-\ell(i)},$ .

\end{lemma}

\begin{proof}

In the following the summation are over \(i,i', 2\le \ell(i),\ell(i')\le r\). 
Trivially, $x' A^r x = \sum_i x_i^2 A_{ii}+  \sum_{i \ne j} x_i A_{ij} x_j$. 
By the first inequality
\[ 
c_1 x' \Lambda^{(r)} x= c_1 \sum_{i} x_i^2 2^{-\ell(i)} < \sum_i x_i^2 A_{ii} < c_2 \sum_i x_i^2 2^{-\ell(i)}= c_2 x' \Lambda^{(r)} x  .
\]

On the other hand
\[
\left|\sum_{i \ne i'} x_i A_{ii'} x_{i'} \right|
 \le c_3  \sum_{i \ne i'}   |x_i| 2^{-1.5 \ell(i)} |x_{i'}|2^{-1.5\ell(i') }  \le c_3  \left(\sum_i |x_i| 2^{-1.5  \ell(i)}\right)^2.
\]
At the first inequality we used the second part of of \eqref{bound}. The second inequality follows upon including the diagonal. By Cauchy-Schwarz, this can be further bounded by 
\[ c_3  \left(\sum_i x_i^2   2^{-\ell(i)} \right) \left(\sum_i 2^{-2\ell(i)}\right) \le  c_3  x' \Lambda x,
\]
where the final inequality follows from   $\sum_i 2^{-2\ell(i)}\le \sum_{i=3}^\infty 2^{-2\ell(i)} =\sum_{j=1}^\infty 2^j 2^{-2j}=1$. The result follows by combining the derived inequalities.
\end{proof}
 We continue with the proof of \cref{prop-prior}. 
Write \(A\) as block matrix
\[
A = \begin{bmatrix} A_{1}  &B' \\B &A_{2}\end{bmatrix},
\]
with \(A_1\) a \(2\times 2\)-matrix, and \(B\), \(A_2\)  defined accordingly. 
By \cref{lem:asymptdiagou}  \[
A_1 = \frac{\sigma^2}2\gamma\begin{bmatrix} \coth(\gamma/2) &\sinh^{-1}(\gamma/2)
\\ \sinh^{-1}(\gamma/2) & \coth(\gamma/2)\end{bmatrix}.\]

Define the \(2\times 2\)-matrix \[\Lambda_1 = c\tfrac{ \sigma^2}2 \gamma \tanh(\gamma/4) \Id,\quad c \in (0,1).\]
where \(\Id\) is the \(2\times 2\)-identity matrix. It is easy to see that \(A_1-\Lambda_1\) is positive definite. 

When $A_2 - \Lambda_2 - B(A_1-\Lambda_1)^{-1} B'$ is positive definite, then it follows from the Cholesky decomposition that  \(A-\Lambda\) is positive definite, where \(\Lambda=\text{diag}(\Lambda_1,\Lambda_2)\) positive definite. 

Note
\[
(B A_1^{-1}  B')_{i,i'} = \sum_{k,k'} B_{ik}(A_1)^{-1}_{kk'}B_{i'k'} \le \left(\sum_{k,k'} (A_1)^{-1}_{kk'}\right) (B_{i,1} \vee B_{i,2}) (B_{i',1} \vee B_{i',2}) 
\]
where 
\[
 \left(\sum_{k,k'} (A_1)^{-1}_{kk'}\right) =  \frac{2}{\sigma^2\gamma} \frac{2}{\sinh^{-1}(\gamma/2) + \coth(\gamma/2)} \le  
 \frac{2}{\sigma^2(1+\gamma)}. 
\]
Therefore
\[
|(B A_1^{-1}  B')_{ii'}| \le  0.020\sigma^2(1+\gamma/4) 2^{-1.5(\ell(i) + \ell(i'))}
\]
Now consider \(
\tilde A = A_2 - \Lambda_2 - B(A_1-\Lambda_1)^{-1} B'
\). By \cref{lem:asymptdiagou} and the bound on $ |(B A_1^{-1}  B')_{ii'} | $ and choosing \(c>0\) in the definition of \(\Lambda_1\) small enough, under the assumption that \(\gamma\le 1.5\), 
\begin{align*}
&0.945\cdot2^{-\ell(i)}\sigma^2/4< \tilde A_{ii} < 1.03\cdot 2^{-\ell(i)}\sigma^2/4.
\end{align*}

and for $i \ne i'$ \(
|\tilde A_{ii'}| \le 0.9415\frac{\sigma^2}42^{-1.5(\ell(i)+\ell(i'))}
\). Therefore by \cref{lem:equivalentmatrices} \(\tilde A - \Lambda_{2}\) is positive definite with diagonal matrix \(\Lambda_{2}\) with diagonal entries \(2^{-\ell(i)}\). 

It follows that \(x'\Lambda x \asymp x'Ax\). This implies that  the small ball probabilities and the mass outside a sieve behave similar under \cref{ass:prior-cov}(B) as when the $Z_{i}$ are independent normally distributed with zero mean and variance $\xi_i^2=\Lambda_{ii}$. As this case corresponds  to \cref{ass:prior-cov}(A) with  $\alpha = \frac12$ for which posterior contraction has already been established, the stated contraction rate under \cref{ass:prior-cov}(B)  follows from Anderson's lemma (\cref{lem:Anderson}).

\subsection{Proof of \cref{thm:contractionforlpnorms}: convergence in stronger norms}

 The linear embedding operator $T\colon L^p(\TT)\to L^2(\TT),x\mapsto x$ is a well-defined injective continuous operator for all $p\in(2,\infty]$. Its inverse is easily seen to be a densely defined, closed unbounded linear operator. Following \cite{KnapikSalomond2014}  we define the modulus of continuity $m$ as  
\[
m(\scr B_n,\eps):=\sup\left\{\|f-f_0\|_p\colon f\in \scr B_n,\|f-f_0\|_2\le\eps\right\}.
\]

Theorem 2.1 of \cite{KnapikSalomond2014} adapted to our case is

\begin{thm}[\cite{KnapikSalomond2014}]
	Let \(\eps_n\downarrow 0, T_n\uparrow\infty\) and  $\Pi$ be a prior on $L_p(\T)$ such that 
	\[
	\EE_0\,\Pi\left(\scr B_n^c\mid X^{T_n}\right)\to 0,
	\]
	for measurable sets $\scr B_{n}\subset L^p(\T)$. Assume that for any positive sequence $M_n$
	\[
	\EE_0\,\Pi\left(b\in \scr B_n\colon \|b-b_0\|_2\ge M_n\eps_n \mid X^{T_n}\right)\to0,
	\]
	then 
	\[
	\EE_0\,\Pi\left(b\in L^p(\T)\colon \|b-b_0\|_p\ge m(\scr B_n,M_n\eps_n)\mid X^{T_n}\right)\to 0.
	\]
\end{thm} 

Note that the sieves $\scr C_{r,t}$ which we define in \cref{subsec:definitionsieves} have by \cref{eq:posterioroftheremainingmassgoestozero} the property $\Pi(\scr C_{r,t}^c\mid X^T)\to0.$ By \cref{lem:modulusofcontinuityupperbound,lem:modulusofcontinuityforlp}, the modulus of continuity satisfies \(m(\scr C_{r,u},\eps_n)\lesssim  2^{r(1/2-1/p)}\eps_n\), for all \(p\in(2,\infty]\), (assume \(1/\infty=0\)), and the result follows.

\appendix



\section{Lemmas used in the proofs}

\begin{lemma}\label{lem:approximationerror}
Suppose \(z\) has Faber-Schauder expansion 
\[
z=z_1\psi_1+\sum_{j=0}^\infty\sum_{k=1}^{2^j}z_{jk}\psi_{jk}.
\]
If\/ $\llbracket z\rrbracket_\beta<\infty$ (with the norm defined in \eqref{eq:betasmoothnessnorm}), then for $r\ge 1$
\begin{equation}\label{appr}
\Big\|z-\sum_{i\in \scr{I}_r} z_i \psi_i\Big\|_{\infty} \le 
\llbracket z\rrbracket_{\beta}\frac{2^{-r\beta}}{2^{\beta}-1}.
\end{equation}
\end{lemma}
\begin{proof}
This follows from
\begin{align*}
\Big\|z-\sum_{i \in \scr{I}_r} z_i \psi_i\Big\|_{\infty}&\le \sum_{j=r+1}^\infty\Big\|\sum_{k=1}^{2^j}z_{jk}\psi_{jk}\Big\|_\infty\\
&= \sum_{j=r+1}^\infty 2^{-j\beta}\max_{1\le k\le 2^j}2^{j\beta}|z_{jk}|
\le \llbracket z\rrbracket_{\beta}
\sum_{j=r+1}^\infty 2^{-j\beta}.
\end{align*}
\end{proof}

\begin{lemma}\label{sec:boundsforinversegammaprior}
If  $X \sim {\rm IG}(A,B)$ then for any $M>0$,
\[
\P(X\ge M)\le \frac{B^A}{\Gamma(A)}M^{-A}.
\]
\end{lemma}
\begin{proof}
This follows 
 from 
\[
\P(X\ge M) \le \frac{B^A}{\Gamma(A)}\int_M^\infty x^{-\alpha-1}\dd x\\
=-\frac{B^A}{\Gamma(A)}\left[ x^{-\alpha}\right]_{x=M}^\infty= \frac{B^A}{\Gamma(A)}M^{-A}.
\]

\end{proof}

\begin{lemma}\label{lem:normalball}
	Let $X \sim \rm N(0,1)$,  $\theta \in \RR$ and $\epsilon > 0$.
	Then
	$$\P(\abs{X - \theta} \le \epsilon) \ge \frac{e^{-\theta^2}}{\sqrt2}\P\left(\abs{X} \le \sqrt2 \epsilon\right) \ge 
	{e^{\log \epsilon - \epsilon^2-\theta^2 +\log\sqrt{\frac{2}{\pi}}}}.$$
\end{lemma}
\begin{proof}
	Note that 
	\begin{align*}
	\int_{\theta-\epsilon}^{\theta+\epsilon} e^{-\frac12 x^2} \dd x&=\int_{-\eps}^\eps e^{-\frac12(x+\theta)^2}\dd x
	%
	\end{align*}
	and
	$$\frac{e^{-\frac12(x+\theta)^2}}{e^{-\theta^2}e^{-\frac12(\sqrt2 x)^2}} = e^{\theta^2 -\frac12(x+\theta)^2  + x^2} = e^{\frac12(x - \theta)^2} \ge 1,$$
	thus 
	$e^{-\frac12(x+\theta)^2}\ge e^{-\theta^2}e^{-\frac12(\sqrt2 x)^2},$ 
	hence 
	$$\int_{\theta-\eps}^{\theta+\eps}e^{-\frac12x^2}\dd x\ge e^{-\theta^2}\int_{-\eps}^\eps e^{-\frac12(\sqrt2 x)^2}\dd x=\frac{e^{-\theta^2}}{\sqrt{2}}\int_{-\sqrt{2}\eps}^{\sqrt{2}\eps} e^{-\frac12u^2}\dd u.$$ 
	Now the elementary bound
	$\int_{-y}^y  e^{-\frac12 x^2 }\ge 2y e^{-\frac12 y^2}$
	gives
	\begin{align*}
	\P(|X-\theta|\le\eps)&=\frac1{\sqrt{2\pi}}\int_{\theta-\eps}^{\theta+\eps}e^{-\frac12x^2}\dd x
	\ge \frac{1}{\sqrt{2\pi}} \frac{e^{-\theta^2}}{\sqrt{2}}\int_{-\sqrt{2}\eps}^{\sqrt{2}\eps}e^{-\frac12u^2}\dd u\\
	&=\frac{e^{-\theta^2}}{\sqrt{2}}\P(|X|\le\sqrt{2}\eps)
	\ge \frac{1}{\sqrt{2\pi}} \frac{e^{-\theta^2}}{\sqrt{2}}2\sqrt{2}\eps e^{-\eps^2}
	=\sqrt{\frac{2}{\pi}}e^{\log\eps -\theta^2-\eps^2}
	\end{align*}
\end{proof}

\begin{lemma}[Anderson's lemma]\label{lem:Anderson}
Define a partial order on the space of $n\times n$-matrices ($n\in\NN\cup\{\infty\}$) by setting 
 $A\le B,$ when $B-A$ is positive definite. 
	\label{lem:anderson} 
	If $X \sim \rm N(0,\Sigma_X)$ and $Y \sim \rm N(0, \Sigma_Y)$ independently with $\Sigma_X \le\Sigma_Y $, then for all symmetric convex sets $C$
	\[
	\P(Y \in C) \le \P(X \in C).
	\]
\end{lemma}
\begin{proof}
 See \cite{Anderson1955}.
\end{proof}

\begin{lemma}\label{lem:boundforcoefficientsintermsofthelinfinitynorm}
	Let 
	\[f=z_1\psi_1+\sum_{j=1}^r\sum_{k=1}^{2^{j}}z_{j,k}\psi_{j,k}.\]
	Then 
	\[
	\sup_{i:\ell(i)\le r}|z_i|\le 2\|f\|_\infty.
	\]
\end{lemma}
\begin{proof}
	Note that \(|z_{1}|=|f(0)|\le2\|f\|_\infty\), and \(|z_{0,1}|=|f(1/2)|\le2\|f\|_\infty\) and inductively, for \(j\ge1\), \(z_{jk}=f\big((2k-1)2^{-(j+2)}\big)-\frac12f\big(2^{-(j+1)}(k-1)\big)-\frac12f\big(2^{-(j-1)}k\big)\), hence \(|z_{jk}|\le 2\|f\|_\infty\).
\end{proof}

\begin{lemma}\label{lem:modulusofcontinuityupperbound}
	Let \(\scr C_{r}\) as in \cref{subsec:definitionsieves}. Then \[
	\sup_{0\neq f\in\scr C_{r}}\frac{\|f\|_\infty}{\|f\|_2}\le \sqrt3\cdot 2^{(r+1)/2}.
	\]
\end{lemma}
\begin{proof}
Let \(f\in \scr C_{r}\) be nonzero. 
	Note that for any constant \(c>0\),
	\begin{align*}
		\frac{\|cf\|_\infty}{\|cf\|_2}=\frac{\|f\|_\infty}{\|f\|_2}.
	\end{align*}
	Hence, we may and do assume that \(\|f\|_\infty=1\). Furthermore, since the \(L^2\) and \(L^\infty\) norm of \(f\) and \(|f|\) are the same, we also assume that \(f\) is nonnegative. 
	
	Let \(x_0\) be a global maximum of \(f\). Clearly \(f(x_0)=1\). Since \(f\) is a linear interpolation between the points \(\{k2^{-j-1}: k=0,1,\ldots, 2^{r+1}\}\), we may also assume that \(x_0\) is of the form \(x_0=k2^{-j-1}\). We consider two cases
	\begin{enumerate}[(i)]
		\item \(0\le k<2^{r+1}\),
		\item \(k=2^{r+1}\).
	\end{enumerate}
	 In case (i) we have that \(f(x)\ge \big(1-2^{r+1}(x-k2^{-r-1})\big)\I_{[k2^{-r-1},(k+1)2^{-r-1}]}(x)\), for all \(x\in[k2^{-r-1},(k+1)2^{-r-1}]\). In case (ii) \(f(x)\ge 2^{r+1}(x-1+2^{-r-1})\I_{[1-2^{-r-1},1]}(x)\), for all \(x\in[1-2^{-r-1},1]\). Hence, in both cases,
\[
	 	\|f\|_2^2\ge 2^{2r+2}\int_0^{2^{-r-1}}x^2\dd x
	 	=\frac13 2^{2r+2}2^{-3r-3}
	 	=\frac132^{-r-1}.
\]
	 Thus
\[
\frac{\|f\|_\infty}{\|f\|_2}\le \frac1{\frac1{\sqrt3}2^{-(r+1)/2}}
=\sqrt 3\cdot 2^{(r+1)/2},\]
	 uniformly over all nonzero \(f\in \scr C_{r,s}\).
 \end{proof}

\begin{lemma}\label{lem:inequality}
	Let \(a_1,a_2,x_1,x_2\) be positive numbers. Then 
	\[
	\frac{a_1+a_2}{x_1+x_2}\le \frac{a_1}{x_1} {\textstyle\bigvee} \frac{a_2}{x_2}.
	\]
\end{lemma}
\begin{proof}
Suppose that the lemma is not true, so there are positive \(a_1,a_2,x_1,x_2\) such that, 
	\begin{align*}
		\frac{a_1}{x_1} {\textstyle\bigvee} \frac{a_2}{x_2}-\frac{a_1+a_2}{x_1+x_2}&=\left[\frac{a_1}{x_1}-\frac{a_1+a_2}{x_1+x_2}\right] {\textstyle\bigvee}\left[\frac{a_2}{x_2}-\frac{a_1+a_2}{x_1+x_2}\right]\\
		&=\left[\frac{\frac{a_1}{x_1}(x_1+x_2)-(a_1+a_2)}{x_1+x_2}\right] {\textstyle\bigvee}\left[\frac{\frac{a_2}{x_2}(x_1+x_2)-(a_1+a_2)}{x_1+x_2}\right]<0.
	\end{align*}
Hence, both terms on the right-hand-side are negative. In particular, this means for the first term that $x_2/x_1< a_2 /a_1$. For the second term this gives $x_1/x_2<a_1 /a_2$. These two inequalities cannot hold simultaneously and we have reached a contradiction. 
\end{proof}

\begin{lemma}\label{lem:modulusofcontinuityforlp}
	Let $\scr C_r$ and \(\scr C_{r,s}\) as in \cref{subsec:definitionsieves}. Then for \(p\in[2,\infty)\),   
	\[
	\sup_{0\neq f\in\scr C_{r}}\frac{\|f\|_p}{\|f\|_2}\le \frac{3^{1/2}}{(p+1)^{1/p}}2^{(r+1)(1/2-1/p)}.
	\]
\end{lemma}
\begin{proof}
Let \(f\in \scr C_{r}\). Just as in proof of  \cref{lem:modulusofcontinuityupperbound} we may assume that \(f\) is nonnegative and \(\|f\|_2=1\). Hence 
\[
	\sup_{0\neq f\in\scr C_{r}}\frac{\|f\|_p}{\|f\|_2}=\left(\sup_{0\neq f\in\scr C_{r}, \|f\|_2=1}\|f\|_p^p\right)^{1/p}
	=\left(\sup_{0\neq f\in\scr C_{r}, \|f\|_2=1}\frac{\|f\|_p^p}{\|f\|_2^2}\right)^{1/p}.\\
\]
Note that 
\begin{align*}
	\|f\|_p^p=\sum_{k=0}^{2^{r+1}-1}\int_{k2^{-r-1}}^{(k+1)2^{-r-1}}f(x)^p\dd x.
\end{align*}
Hence, by repeatedly applying  \cref{lem:inequality} 
\begin{align*}
\frac{\sum_{k=0}^{2^{r+1}}\int_{k2^{-r-1}}^{(k+1)2^{-r-1}}f(x)^p\dd x}{\sum_{k=0}^{2^{r+1}-1}\int_{k2^{-r-1}}^{(k+1)2^{-r-1}}f(x)^2\dd x}&\le \bigvee_{\tiny\begin{array}{c}k\in\{0,\ldots,2^{r+1}-1\}\\\exists x\in (k2^{-r-1},(k+1)2^{-r-1}):f(x)\neq0\end{array}}\frac{\int_{k2^{-r-1}}^{(k+1)2^{-r-1}}f(x)^p\dd x}{\int_{k2^{-r-1}}^{(k+1)2^{-r-1}}f(x)^2\dd x}.
\end{align*}

 Note that \(f\) is a linear interpolation between the points \(k2^{-r-1},k\in\{0,1,\ldots,2^{r+1}\}\). 
 
Now study affine functions \(g:[0,2^{-r-1}]\to \re\) which are positive. A maximum of \(g\) is attained in either \(0\) or \(2^{-r-1}\). Without lose of generality it is attained in \(0\). Using scaling in a later stadium of the proof, we assume for the moment that \(g(0)=1\). Hence \(a:=g(2^{-r-1})\in[0,1]\).
  Note that 
  \[
  g(x)=1-(1-a)2^{r+1}x.
  \]
  
  When \(a=1\), \(\|g\|_p=\|g\|_2=1\). Now consider \(a<1\),
  \begin{align*}
  	\int_0^{2^{-r-1}}g(x)^p\dd x&=\int_0^{2^{-r-1}}\big[1-(1-a)2^{r+1}x\big]^p\dd x.
  \end{align*}
  Let \(y=-x+\frac{2^{-r-1}}{1-a}\) then \(x=-y+\frac{2^{-r-1}}{1-a}\) and \(\dd x=-\dd y\). Hence 
  \begin{align*}
  	\int_0^{2^{-r-1}}g(x)^p\dd x&=\int_{-2^{-r-1}+\frac{2^{-r-1}}{1-a}}^{\frac{2^{-r-1}}{1-a}}(1-a)^p2^{rp+p}y^p\dd y\\
  	&=2^{-r-1}\frac1{p+1}\left[\frac{1-a^{p+1}}{1-a}\right].  \end{align*}
  	
  	Note that for a constant \(c>0\) and a function \(h\),
\[
  		\frac{\|ch\|_p^p}{\|ch\|_2^2}=\frac{c^p\|h\|_p^p}{c^2\|h\|_2^2}
  		=c^{p-2}\frac{\|h\|_p^p}{\|h\|_2^2}.\]
  	Let 
  	\[
  	c^2=3\cdot2^{r+1}\frac{1-a}{1-a^3}.
  	\]
  	Hence \(cg\) has \(L^2\)-norm one and 
  	\begin{align*}
  		\|cg\|_p^p&=c^p\|g\|_p^p\\
  		&=\left(3\cdot2^{r+1}\frac{1-a}{1-a^3}\right)^{\frac p2}2^{-r-1}\frac1{p+1}\left[\frac{1-a^{p+1}}{1-a}\right]\\
  		&=\frac{3^{p/2}}{p+1}2^{(r+1)(p/2-1)}(1-a)^{p/2-1}(1-a^3)^{-p/2}(1-a^{p+1}).
  	\end{align*}
  	The maximum is attained for \(a=0\), then
  	
  	\[
  	\|cg\|_p^p=\frac{3^{p/2}}{p+1}2^{(r+1)(p/2-1)}
  	\]
  	
  	Hence 
  	\[
  	\|cg\|_p=\frac{3^{1/2}}{(p+1)^{1/p}}2^{(r+1)(1/2-1/p)}
  	\]
  	and the result follows, using that \(\|f\I_{(k2^{-r-1},(k+1)2^{-r-1})}\|_2^2\le \|f\|_2^2\) and that for \(0<c'<c\), 
\[\frac{\|c'g\|_p^p}{\|c'g\|_2^2}=\frac{(c'/c)^p}{(c'/c)^2}\frac{\|cg\|_p^p}{\|cg\|_2^2}\le \frac{\|cg\|_p^p}{\|cg\|_2^2}.\]   
\end{proof}


%
		 
\section{Acknowledgement}
This work was partly supported by the Netherlands Organisation for Scientific Research (NWO) under the research programme ``Foundations of nonparametric Bayes procedures'', 639.033.110 and by the ERC Advanced Grant ``Bayesian Statistics in Infinite Dimensions'', 320637.

\bibliographystyle{apalike}
\bibliography{lit}
\end{document}